\UseRawInputEncoding
\documentclass[12pt]{amsart}
\usepackage{amsmath, amssymb,latexsym }
\usepackage{verbatim}
\usepackage{enumerate}

\usepackage{color}

\usepackage{graphicx}

\def\XXint#1#2#3{{\setbox0=\hbox{$#1{#2#3}{\int}$ }
\vcenter{\hbox{$#2#3$ }}\kern-.6\wd0}}

\renewcommand{\epsilon}{\varepsilon}

\newcommand{\ang}[1]{{\left\langle{#1}\right\rangle}}

\renewcommand{\Pi}{\varPi}
\renewcommand{\phi}{\varphi}

\renewcommand{\epsilon}{\varepsilon}

\usepackage[active]{srcltx}
\usepackage{graphicx}

\def\L2R{L_{\text{Rest}}^2}

\newcommand{\RR}{\mathbb{R}}

\def\text{\textstyle}

\def\text{\textstyle}

\def\calP{\pcal}

\def\calU{\ucal}

\newcommand{\T}{{\mathbf T}^m}

\newcommand{\wt}{\widetilde}

\newcommand{\R}{{\mathbb R}}

\newcommand{\Q}{{\mathbb Q}}
\newcommand{\Z}{{\mathbb Z}}

\newcommand{\diag}{{\operatorname{diag}}}

\renewcommand{\phi}{\varphi}

\newcommand{\bcal}{\mathcal{B}}

\newcommand{\dcal}{\mathcal{D}}
\newcommand{\ecal}{\mathcal{E}}
\newcommand{\fcal}{\mathcal{F}}

\newcommand{\ical}{\mathcal{I}}

\newcommand{\mcal}{\mathcal{M}}

\newcommand{\pcal}{\mathcal{P}}
\newcommand{\rcal}{\mathcal{R}}
\newcommand{\ocal}{\mathcal{O}}

\newcommand{\ucal}{\mathcal{U}}

\def	\Span   {{\operatorname{span}}}

\def    \t  {{\mathfrak t}}

\def    \Z  {{\mathbb Z}}
\def    \R  {{\mathbb R}}

 \def \L{{\bf L}}

\def\text{\textstyle}

\newcommand{\SE}{\mathcal{E}}

\def\sq2{\sqrt{2}}

\def\t2{{\mathbb T}^2}
\def\s2{{\mathbb S}^2}

\def\T{\mathbb{T}}
\def\R{\mathbb{R}}
\def\Z{\mathbb{Z}}

\def\sq2{\sqrt{2}}

\def\t2{{\mathbb T}^2}
\def\s2{{\mathbb S}^2}

\def\T{\mathbb{T}}
\def\R{\mathbb{R}}
\def\Z{\mathbb{Z}}

\def\hto0{\xrightarrow{\hbar\to 0}}

\def\rto0{\xrightarrow{r\to 0}}

\renewcommand{\phi}{\varphi}

\newtheorem{theo}{Theorem}[section]
\newtheorem{lem}[theo]{Lemma}
\newtheorem{prop}[theo]{Proposition}
\newtheorem{cor}[theo]{Corollary}

\newtheorem{rem}[theo]{Remark}

\theoremstyle{definition}
\newtheorem{defn}[theo]{Definition}

\numberwithin{equation}{section}
\usepackage{hyperref}

 \title{Self-focal points of ellipsoids of dimension $\geq 3$ }

\author{Sean Gomes and Steve Zelditch  }
\address{Department of Mathematics, Northwestern  University, Evanston, IL 60208, USA}

\email{zelditch@math.northwestern.edu}

\thanks{Research partially supported by NSF grant   DMS-1810747 }

\begin{document}

\maketitle

\begin{abstract} A self-focal point of a Riemannian manifold $(M,g)$  is a point $p$ so that every geodesic starting from $p$ returns to $p$ at some
positive time. There is a first return map $\Phi_p: S^*_p M \to S^*_p M$ taking the initial direction of the geodesic to its terminal direction at $p$ at the
first return time. The self-focal point $p$ is a `pole' if $\Phi_p$ is the identity map $Id$. The  map $\Phi_p$ is said to be `twisted' if it is not the identity map $Id$  (or, infinitesimally twisted if $D_{\xi} \Phi_p$ is not the identity 
for all $\xi \in S^*_p M$).  On the standard sphere, all points are poles. On a surface of revolution, the poles are poles. On a tri-axial two-dimenional  ellipsoid, the 4 umbilic points are
self-focal but are not poles.  Little is known about  the existence or non-existence  of self-focal points for manifolds of dimension
$\geq 3$. In this article, 
we prove that ellipsoids of dimension $\geq 3$ with at least 4 distinct axes have no self-focal points. Ellipsoids with $\leq 2$ distinct axes always have self-focal points. Certain
ellipsoids with $3$ distinct axes have non-polar  self-focal points. The main technique is Moser's Lax-pair approach
to integrability of geodesic flows on ellipsoids. 

Self-focal points play an important role in the study of $L^{\infty}$-norms of Laplace eigenfunctions. The results on non-existence of self-focal
points of $\geq 4$ axial ellipsoids imply that their eigenfunctions never achieve maximal sup norm growth. 
\end{abstract}

This article is concerned with the existence of {\it self-focal points}  and of {\it poles} of  $(n-1)$-dimensional ellipsoids, \begin{equation} 
\label{ECALA}  \ecal_A:  = \{x \in \R^{n} :  \langle A^{-1} x, x \rangle = 1\}. \end{equation}
 Here,  $A$ is a real  $n \times n$ symmetric matrix with positive eigenvalues $\alpha_1 \leq  \cdots \leq \alpha_n$; it is always assumed  that $A$ is diagonal and that eigenvalues are enumerated in non-decreasing order.
$\ecal_A$ is called:
\begin{itemize}

\item  {\it multi-axial} if $A$ if all of its  eigenvalues $\{\alpha_j\}_{j=1}^n$ are distinct, i.e. have multiplicity one. 
\item {\it k-axial} if $A$ has exactly $k$ distinct eigenvalues; the  multiplicities of the distinct eigenvalues  
$\alpha_1^* <  \cdots < \alpha_k^* $ are denoted by $(m_1, \dots, m_k)$.
\end{itemize}

We are interested in the existence of special points of $\ecal_A$: 
\begin{itemize}

\item A point $x_0 \in \ecal_A$ is {\it self-focal} if there exists at time $T > 0$ so that $\exp_{x_0} T \xi = x_0$
for all $\xi \in S^*_{x_0} \ecal_A$. I.e. every geodesic `loops back'  starting at $x_0$ loops back to $x_0$ at time $T$.
\item A point $x_0 \in \ecal_A$ is a {\it pole} if there exists $T > 0$ so that $G^T(x_0, \xi) = (x_0, \xi)$, where $G^t: S^*\ecal_A
\to S^*\ecal_A$ is the geodesic flow. That is, every geodesic through $x_0$ is smoothly closed. 

\end{itemize}

 It is well-known that when $n =3$, a 2-dimensional  tri-axial ellipsoid
has $4$ umbilic points $\{x_j\}_{j=1}^4$. The term `umbilic' means that the second fundamental form of $\ecal_A$ is a scalar multiple
of the metric at $x_j$. Jacobi proved that umbilic points of an ellipsoid are self-focal \cite{J,Kl95}. However, they are not poles: with the exception 
of  two directions at  $x_j$, each geodesic from $x_j$ returns to $x_j$ but the terminal direction is different from the initial direction.   
On the other hand, a $2$-axial three dimensional ellipsoid is a surface of revolution, and the fixed points of the rotation action are poles. 
A $1$-axial ellipsoid is a standard sphere, and all of its points are poles.

The problem we raise is the existence of self-focal points and poles on higher dimensional ellipsoids. It is obvious that all points of   spheres $S^{n-1} = \ecal_I$ of any dimension are poles, and that $2$-axial ellipsoids with multiplicities $(1, n-2)$  in $\R^n$ carry an $SO(n-1)$ action whose
fixed points are poles.   It does not seem to have been known whether there exist self-focal points for any other higher dimensional ellipsoids $\ecal_A$, or of whether they are poles.  The first result answers the question if $A$ is at least $4$-axial.

\begin{theo}\label{THEO4DISTINCT} For any symmetric positive definite matrix $A$ with at least four distinct eigenvalues,  $\ecal_A$ has no self-focal points. \end{theo}

The  extreme cases are  ellipsoids of revolution with $SO(n-1)$ symmetry,  including spheres with
$SO(n)$ symmetry.
In the  cases where $A = Id$ or when $A$ has $SO(n-1)$,  the fixed points of the symmetry
group are poles. Hence we assume in what follows that $A$ has $\geq 3$ distinct eigenvalues. The multi-axial  assumption underlies almost all known results on the geodesic flow on ellipsoids (cf.  \cite{M80,K80, K85, Au,  Kl95}).

This  problem is motivated by
recent results on   $L^{\infty}$ norms of eigenfunctions $\phi_{\lambda}$ of the Laplacian of eigenvalue $\lambda^2$. There exists a universal upper bound $||\phi_j||_{L^{\infty}} \leq C_{M,g} \lambda_j^{\frac{n-1}{2}} ||\phi_j||_{L^2}$ on the $L^{\infty}$ norm
of an $L^2$-normalized eigenfunction. When the bound is achieved by a sequence $\{\phi_{\lambda_j}\}$ of eigenfunctions,  the points where  $|\phi_{\lambda_j}(x)|$ is of maximal growth in $\lambda_j$ 
are known to be self-focal points (see Section \ref{RELATED} and  \cite{Z18} for background). In fact,  all known examples where maximal growth is achieved are poles and in \cite{SZ16} it is proved that for real analytic surfaces they must be poles. For instance, eigenfunctions of a  tri-axial
ellipsoid do not achieve maximal $L^{\infty}$-norm growth but do take their asymptotic maxima at the umbilic points. 
 Eigenfunctions on a
2-axial ellipsoid (or any surface of revolution) achieve maximal $L^{\infty}$-norm growth and have their maxima at the poles.
By combining the results of \cite{SZ16} with Theorem \ref{THEO4DISTINCT}, we obtain a new result about maximal
sup norm growth of eigenfunctions on ellipsoids:  \begin{cor}\label{EFCOR} For any symmetric positive definite matrix $A$ with at least four distinct eigenvalues,  any orthonormal basis of Laplace eigenfunctions $\{\phi_j\}_{j=1}^{\infty}$ 
satisfies $||\phi_j||_{L^{\infty}}  = o(\lambda_j^{\frac{n-1}{2}}). $

\end{cor} 
See Section \ref{RELATED} for further comments on applications to eigenfunctions.

 The eigenfunction results  raise the question whether there exist compact real analytic Riemannian manifolds $(M,g)$ of dimension
 $\geq 3$ which possess self-focal points, in particular non-polar self-focal points. 
 Compact Riemannian manifolds with self-focal points are rare, and those with  non-polar self-focal points are rarer. Ellipsoids would seem to be
 a natural hunting ground for such points.    Although Jacobi's theorem that umbilic points of an ellipsoid are self-focal is very well known, we are not aware of any prior studies on its possible generalizations to higher dimensions  (see \cite{DDB,DD07} for detailed analyses of geodesics and umbilic curves 
 on three-dimensional ellipsoids and some generalizations to  higher dimensional ellipsoids; see also \cite{IK04,IK10,IK19} for results
 on cut-loci and conjugate loci). Moreover, we are not aware of any known examples 
of compact Riemannian manifolds of dimension $ \geq 3$ which possess non-polar  self-focal points, and in particular no real-analytic examples.

In the case of  tri-axial ellipsoids, self-focal points are umbilic points, i.e. points where the principal curvatures coincide. One may ask if
there exists a relation between umbilicity and self-focality in higher dimensions. 
 Self-focality is an intrinsic dynamical property of the geodesic
flow and  the natural projection $\pi: T^* M \to M$. The umbilic property is an extrinsic property
defined by the second fundamental form of the hypersurface. It is easy to give examples of non-umbilic self-focal points: Since every point of a Zoll manifold is a self-focal point, but seldom are the points umbilic,  perfect self-focal points of embedded hypersurfaces are not generally umbilic.

Theorem \ref{THEO4DISTINCT} raises the question whether there exist $3$-axial ellipsoids with self-focal points and, in particular,
non-polar self-focal points.  This question turns out to be more complicated than the case of $\geq 4$ axial ellipsoids, because it turns
out there do exist examples with self-focal points. Classifying all the possible types would be lengthy and complicated,  and we only
present a partial result. To state the result, we need some further notation. A
self-focal point $x_0$ is a point such that  the geodesic flow  \begin{equation} \label{GEOFLOW} G^t: S^* \ecal_A \to S^*\ecal_A \end{equation} 
takes the unit cosphere $S^*_{x_0} \ecal_A$ to itself at some time, which we normalize to be $2 \pi$. The first return map on directions is
defined by,
 \begin{equation} \label{PhiDEF} \Phi_{x_0} = G^{2 \pi} : S^*_{x_0} \ecal_A \to
S^*_{x_0} \ecal_A, \end{equation}
taking the initial direction $\xi$ at a  self-focal point  to the terminal direction $G^{2 \pi}(x_j, \xi)$ of the loop. The self-focal point
is  `polar' if $\Phi_{x_0} = Id$. We say that $\Phi_{x_0}$ is `twisted' if, for all $\xi$, $D_{\xi} \Phi_{x_0}$ never has $1$ as an eigenvalue;
otherwise we say it is `untwisted. For instance, the first return map at an umbilic point of a two-dimensional tri-axial ellipsoid is twisted.

 \begin{prop} \label{3AXESPROP} Suppose that $\ecal_A \subset \R^n$ is a $3$-axial ellipsoid with axis multiplicities
 $(1, n-2, 1)$. Then $\ecal_A$ has a `twisted'  self-focal point of the form $u = (x_1,\vec 0, x_2)$, where $(x_1, 0, x_2)$
 is the umbilic point of the two-dimensional ellipsoid with the same axis lengths.
 
   \end{prop}
   
   It is rather complicated to study all possible tri-axial ellipsoids of general dimensions, and there do not seem to exist
   any prior studies of them in dimension $> 3$. Three dimensional ellipsoids with multiple axes are studied in \cite{DD07}.
   In Section \ref{3AXESSECT} we also discuss the case of $(1, n-2, 1)$ ellipsoids but even in the three-dimensional
   case we leave it open whether it has self-focal points. If any self-focal point exits, its  $SO(n-2)$-orbit consists of
   self-focal points, and they are almost un-twisted. We point out in Lemma \ref{ROSOLEM}  that the question whether a (2, 1,1) three-dimensional tri-axial ellipsoid
   has a self-focal point is equivalent to the question whether the usual umbilic points of the corresponding two-dimensional tri-axial
  ellipsoid are self-focal for every reduced Rosochatius Hamiltonian flow.
  
\subsection{Complete integrability and Lagrangian torus foliations}
The proof of Theorem \ref{THEO4DISTINCT}  is based on an analysis of  special properties of geodesic flows of ellipsoids that may have an independent interest. 
It is a classical fact (first proved by C. G. J.  Jacobi circa 1840 \cite{J}) that  multi-axial ellipsoids (with distinct axes)  have completely integrable geodesic flow. At first sight, 
it would appear that complete integrability  is incompatible with existence of non-polar self-focal points in dimensions $n-1 \geq 3$.
As reviewed in Section \ref{MMSECT},  complete integrability of \eqref{GEOFLOW} means that there exists    a Hamiltonian $\R^{n-2}$ action on $T^* \ecal_A \backslash 0$ commuting
with the geodesic flow, i.e. there exist $n-2$ Poisson commuting `invariant functions' $\{p_1, \dots, p_{n-2}\}$ in addition to the metric norm function $|\xi|_g^2 = \sum_{i,j} g^{ij}(x) \xi_i \xi_j$ generating the geodesic flow \eqref{GEOFLOW}.

The geodesic flow induces a $G^t$-invariant Lagrangian
submanifold,   the flow-out of $S_{x_0}^* \ecal_A$,  \begin{equation}
\label{FLDEF}\Lambda_x:= FL(S^*_x \ecal_A): = \bigcup_{t \in \R} G^t (S^*_x \ecal_A)  \subset S^*\ecal_A, \end{equation} diffeomorphic to $S^1 \times S^{n-2}$.
The flow-out is the `mapping torus' of the first return map,
  It follows from the Liouville-Arnold theorem (see Theorem \ref{LATH}) that
$ S^* \ecal_A$ has a (singular) foliation by invariant tori. Since \eqref{FLDEF} would  an invariant Lagrangian
submanifold of $S^*\ecal_A$ of topological type $S^1 \times S^{n-2}$, it would be rather surprising to find it in addition to the foliation
by Lagrangian tori. Indeed, we prove that it cannot exist. The subtlety is that the orbits through $(x, \xi) \in S_x^*\ecal$ at a self-focal
point could  all be singular and \eqref{FLDEF} might be a union of such singular orbits.

 In the modern treatment of
J. Moser \cite{M80}, the  additional  `constants of the motion'  are essentially  the 
eigenvalues of J. Moser's Lax  matrices, \begin{equation} \label{LDEF} L(x, \xi) = P_{\xi} (A - x \otimes x^*) P_{\xi},\;\; (x, \xi) \in \R^n \times \R^n.  \end{equation}
  Here, $P_{\xi} = I - \xi \otimes \xi^*$ is the orthogonal projection onto the orthogonal subspace to $\xi$.  Clearly,  $L(x, \xi) \xi =0$ and if
  $(x, \xi) \in S^* \ecal_A$ then $P_{\xi} {\bf n}_x  =0$ where ${\bf n}_x$ is the unit normal at $x$. The Lax matrix  \eqref{LDEF} 
has $n-2$ other  eigenvalues $\{\lambda_j(x, \xi)\}_{j=1}^{n-2}$, which   are invariants of the geodesic flow$L(x, \xi)$. Elementary symmetric
functions of the eigenvalues can   be assembled into a moment
map $\pcal_e= (p_1, \dots, p_{n-2}): S^* \ecal_A \to \R^{n-2}$, which we call the {\it eigenvalue moment map} (see \eqref{pcaldef}). The full moment map 
is then $\pcal = (|\xi|_g^2, \pcal_e):S^* \ecal_A \to \R^{n-1}$.  It requires a substantial analysis of the
Moser eigenvalue invariants to prove that this cannot happen.

\subsection{\label{OUTLINESECT} Outline of the proofs}  

 We begin by proving  that a multi-axial ellipsoid $\ecal_A$ 
of dimension $\geq 3$ has no self-focal points of any kind. 
Our first result is a simpler version of Theorem \ref{THEO4DISTINCT}.
\begin{theo}\label{THEODISTINCT}  Suppose that $\ecal_A$ is multi-axial, i.e. that $A$  has simple eigenvalues. Then $\ecal_A$ has no self-focal points. \end{theo}

When  $A$ has simple eigenvalues, it has been proved  that the  eigenvalues/eigenvectors
of \eqref{LDEF} have an interpretation in terms of geometry of $n-2$ confocal quadrics to $\ecal_A$ (Chasles' theorem; see \cite{M80, K80, Au}). Namely,
the eigenvalue parameters $\lambda_j$ select $n-2$ confocal quadrics  ${\mathfrak A}_{\lambda_j}$ \eqref{QUADRICDEF}  with the property that tangent lines to geodesics are tangent to each quadric. Moreover, the
quadrics are invariant as the  covector moves  along orbits of the geodesic flow. To our knowledge,  no such interpretation is known 
for ellipsoids with multiple axes.

To prove Theorem \ref{THEODISTINCT}, we study the restricted moment map  \begin{equation} \label{PCALx0} \pcal |_{S_{x_0}^* \ecal_A} :
{S_{x_0}^* \ecal_A} \to \R^{n-2}. \end{equation} 
We first prove that at a self-focal point, none of the orbits of the Hamiltonian $\R^{n-1}$-action can be `Liouville tori' (i.e. compact,
regular orbits).

\begin{lem}\label{REGintro}  Suppose that $x_0$ is a self-focal point of a multi-axial  ellipsoid, and $(x_0, \xi) \in S^*_{x_0} M$.
Then either $(x_0, \xi)$ is a singular point of the moment map or it is a regular point with a non-compact orbit.  \end{lem}
An initial  proof is given in  Proposition \ref{REG} and \ref{PSFSING} when $A$ has distinct
eigenvalues.

Both of the alternatives in the statement may occur.  For instance, in    the case of an  umbilic point  $x_0$ for a 2-dimensional trix-axial ellipsoid,
 $S^*_{x_0} \ecal_A$ consists    of four  orbits  (two non-compact cylinders and two circles, corresponding to the smoothly 
closed geodesics). The moment map is singular along the closed geodesics but is regular along the cylinders. 
The  first return map $\Phi_{x_0}$ at the umbilic point is a diffeomorphism of $S^1$ with  two fixed points, and is a north-pole-
south-pole map taking all other points to the north pole. 

The proof  of Lemma \ref{REGintro} is in part  an application of a theorem due to Schutz \cite{Sch72}  and Klingenberg \cite{Kl95}.  Their results use that the
geodesic flow of a multi-axial ellipsoid is {\it Kolmogorov non-degenerate} (see Definition \ref{KNDDEF}). However,  to prove Lemma
\ref{REGintro}, we need 
the somewhat different non-degeneracy condition that the geodesic flow is {\it iso-energetically non-degenerate} (Definition \ref{KNDDEF}). 
In general, these two conditions are independent. However, we prove the following general result, which allows us to infer a sufficient
degree of iso-energetic non-degeneracy to prove Lemma \ref{REGintro}.

\begin{prop}\label{MAINNDintro}  Suppose that $(M,g, H)$ is a Kolmogorov non-degenerate integrable system with Hamiltonian $H(x, \xi)$ that is homogeneous
of degree $2$ in $\xi$. Then, there exists an open dense set in $S^*M = \{H=1\}$ on which the system is iso-energetically non-degenerate. \end{prop}

The next step is to embed  $(S^*_{x_0} M, \Phi_{x_0})$ as an invariant Lagrangian manifold for  the nonlinear Poincar\'e map of a local symplectic
transversal  $ S \subset S^*M$ to the geodesic flow in $S^*M$.  Natural symplectic transversals are provided by the sets $S^*_H \ecal_A$ of unit covectors to $\ecal_A$ with
footpoint on a hypersurface  $H$.  To qualify as a transversal, the Hamilton vector field $H_{|\xi|_g^2}$ of the co-metric norm function $|\xi|_g^2 = \sum_{i,j} g^{ij}(x) \xi_i \xi_j$
on $T^*\ecal_A$ must be transverse to $S^*_H \ecal_A$. In terms of the geometry of the ellipsoid, geodesics which intersect $H$ can only have order of
contact $\leq 1$. Hence we may let $H \subset \ecal_A$ be a (piece of open) hypersurface through $x_0$ with non-degenerate second fundamental form.
For such a local transversal there exists a partially-defined first return map $\fcal: S \to S$ defined at points where the time of return to $S$ is finite. We only
use $\fcal$ in arbitrarily small neighborhoods of $S^*_{x_0} \ecal_A$ and, for any $n \in \Z$,  there exist neighborhoods for which the first $n$ return times
are finite.

Having chosen $H$, we consider the nonlinear Poincar\'e map,
\begin{equation} \label{PMAPDEF} \fcal: S^*_H \ecal_A \to S^*_H \ecal_A, \end{equation}
where it is understood that the map is only defined on the domain where the first return time is finite. Evidently, $S_{x_0}^*\ecal_A \subset S^*_H \ecal_A$ is an isotropic submanifold of a symplectic manifold of half the dimension $2 (n-1) - 2$, hence 
it is a Lagrangian submanifold.  Moreover, $S_{x_0}^*\ecal_A$ is an invariant
set for $\fcal$ and, 
$$\fcal | _{S^*_{x_0} \ecal_A} =  \Phi_{x_0}: S_{x_0}^*\ecal_A \to S_{x_0}^*\ecal_A. $$

We now intersect the (singular) foliation of $S^*\ecal_A$ by the $\R^{n-1}$-orbits of the Hamiltonian action with the transversal $S^*_H \ecal_A$ to obtain a (singular) foliation of
$S^*_H \ecal_A$ with leaves of generic dimension $n-2$, which we call the {\it reduced} foliation.

\begin{prop} \label{SINGLEAF} If $\ecal_A$ is multi-axial, then   $S^*_{x_0} \ecal_A$ lies on a single level set of the eigenvalue moment map \eqref{PCALx0}, and   is a leaf of the reduced foliation.  Moreover, $\Lambda_{x_0}$ is the closure of a regular
orbit. 
 \end{prop} 

As is well-known, and reviewed in Section \ref{MOSERSECT}, the eigenvalues $\lambda_j$ of the moment map at a point $(x, \xi)$ specify confocal quadrics $Q_{\lambda_j}$ 
to $\ecal_A$ with the property that 
the tangent lines of  the geodesic $\gamma_{x, \xi}(t) $ with initial data $(x, \xi)$ are tangent to $Q_{\lambda_j}$. These quadrics degenerate when, and only when, $\lambda_j$
is an eigenvalue of $A$.

\begin{prop} \label{DEGPROP}  The constant value $\vec \lambda$ of  \eqref{PCALx0}  is a value $\vec \lambda$ for which all the quadrics ${\mathfrak A}_{\lambda_j}$ are degenerate. Such values are $\lambda_j = a_j$.

\end{prop}

We then complete the proof of Theorem \ref{THEODISTINCT} in two different ways. One is to use the
 $n-2$ confocal quadrics associated to each point $(x_0, \xi) \in S^*_{x_0} \ecal_A$ together with the fact that 
\eqref{PCALx0} is a constant map. In this case,  the lines through every $(x_0, \xi) \in S^*_{x_0} \ecal_A$ must be tangent to the same $n-2$ quadrics. Yet, these
lines span the tangent plane $T_{x_0} \ecal_A$, which would therefore have to be tangent to the quadrics. Hence, $T_{x_0} \ecal_A$ would have to
be the tangent space to each ${\mathfrak A}_{\lambda_j}$ at the point of contact  of each line $x_0 + s \xi, \xi \in S^*_{x_0} \ecal_A$ with ${\mathfrak A}_{\lambda_j}$. This gives a contradiction. 
The second of these proofs again exploits the fact that \ref{PCALx0} is a constant map. By computing formulae for the variation of eigenvalues along smooth curves in $S^*_{x_0}$, we deduce that $A-x_0\otimes x_0$ must restrict to a scalar operator on the subspace $T^*_{x_0}\SE_A$, and so must have an eigenvalue of multiplicity at least $n-1$. It follows that $A$ must have an eigenvalue of multiplicity at least $n-2$, which is a contradiction for $n>3$.

Theorem \ref{THEO4DISTINCT} is proved in Section \ref{4AXES}. 
The main idea is to use the symmetries to  reduce the statement to the case of multi-axial ellipsoids.
That is, we use the isometry group of $\ecal_A$ to move any putative
self-focal point around, and in particular to move it to a coordinate slice. It then becomes a self-focal points for a lower-dimensional multi-axial ellipsoid of dimension $\geq 3$. We deploy Theorem \ref{THEODISTINCT} to obtain a contradiction. 

Proposition \ref{3AXESPROP} is proved in Section \ref{3AXESSECT}.  Further observations on tri-axial ellipsoids in higher dimensional
ellipsoids are recorded there as well. For instance, to disprove the existence of self-focal points on tri-axial ellipsoids of general dimension with certain multiplicity
configurations, it would suffice to disprove the existence on four-dimensional ellipsoids with multiplicities $(1, 2, 1)$. The geodesic
flow of such ellipsoids reduces to the so-called Rosochatius Hamiltonian system, i.e. geodesic motion on a two-dimensional ellipsoid  in an inverse square potential.  Although
this Hamiltonian system is known to be completely integrable, with a Lax pair, it has not been analyzed very thoroughly and we leave it
as an open problem whether Rosochatius Hamiltonian systems have self-focal points. They would necessarily occur at the umbilic points
of the reduced two-dimensional ellipsoid.

  \subsection{\label{RELATED} Related problems}
  
 Apart
  from the case of elllipsoids, almost nothing seems to be known about 
  the existence of self-focal points on manifolds of dimension $\geq 3$, or on the existence of  $G^t$-invariant Lagrangian submanifolds such as \eqref{FLDEF} in $S^*M$.  Although our methods use special facts about ellipsoids, in particular their complete integrability by the method
  of Lax and Moser, it is hoped that experience with ellipsoids will shed light on the general case.

Let us provide some more background to Corollary \ref{EFCOR}. 
   As mentioned there, part of the  motivation to study self-focal points is their  important in the study of $L^{\infty}$ norms of eigenfunctions.
    It is conjectured that   the universal bounds
can only be achieved at polar self-focal points. The conjecture is proved in \cite{SZ16} for real analytic surfaces but remains open for general smooth
surfaces and in higher dimensions.  The eigenfunction analysis only goes so far, and then we are left with a geometric problem. We do not know of {\it any}   examples  of compact Riemannian manifolds of dimension $\geq 3$ possessing non-polar self-focal points $x$, i.e. for which the looping geodesics are {\it not} closed, i.e. for which $\Phi_x \not= Id$. Thus it is possible that geometric analysis of self-focal points  can   simplify the analysis of eigenfunctions by constraining the types of Riemannian manifolds possessing them.

The small-oh estimate of Theorem \ref{EFCOR}  can probably be improved to a logarithmic improvement of the form, $||\phi_j||_{L^{\infty}}  = O(\frac{\lambda_j^{\frac{n-1}{2}}}{\sqrt{\log \lambda_j}})$. It appears that the machinery of \cite{CG19}
together with the geometric results in this article are sufficient to prove this logarithmic improvement. See also \cite{GT20} for the
estimate on a two-dimensional tri-axial ellipsoid as well as references to the fact that the estimate is sharp.
  In fact, it seems reasonable to ask if there exist power law improvements. 
  
  \subsection{Acknowledgements} We would like to thank Peter Topalov for his comments on iso-energetic non-degeneracy
  and its relations to Kolmogorov non-degeneracy

\section{\label{MMSECT} Integrable systems and moment maps }
In this section, we review basic concepts of integrable systems: (i) regular compact and non-compact points, (ii) singular points, (iii) frequency map and 
action-angle variables; (iv) resonant frequencies. All of these notions are used in the proofs of the main results. We also need the less-standard result that a non-compact orbit can only occur on a singular level set of the moment map. Since we were unable to find it in the literature, we given a proof in Lemma \ref{NONCPTPROP}.

The geodesic flow $G^t$ of a Riemannian manifold $(M,g)$ of dimension $d$ is the Hamiltonian flow of the metric norm function $H(x, \xi) = |\xi|_g^2$ on $T^* M$.
the Hamiltonian vector field of a smooth function $p$ is denoted $H_{p}(x, \xi) = \sum_{j= 1}^d (\frac{\partial p}{\partial \xi_j} \frac{\partial}{\partial x_j} - 
\frac{\partial p}{\partial x_j} \frac{\partial}{\partial \xi_j} )$, and its flow is denoted by $\Phi_t = \exp t H_p$.
  The geodesic flow of an ellipsoid $\ecal$ is completely integrable, i.e. there
exists a  Hamiltonian $\R^n$ action on $T^* \ecal$
commuting with the geodesic flow. In this section, we review definitions relevant to general  completely integrable systems. In the next section,
we review Moser's approach to the complete integrability of the geodesic flow of an ellipsoid. We then study the moment map and its singular
points for Moser's moment map.

  A completely integrable system in dimension $d$ is defined by an abelian
  subalgebra
\begin{equation}
\label{p} {\mathfrak p} =  \R  \{p_1, \dots, p_d\}\;\subset \;
(C^{\infty} (T^*M - 0), \{ , \}).
\end{equation} Here,  $\{, \}$ is the standard Poisson
bracket. We assemble the generators into the  moment map
\begin{equation} \label{MM} {\mathcal P} = (p_1, \dots, p_d): T^*M \rightarrow
B \subset\R^d. \end{equation}  
We always assume that  moment maps ${\mathcal P}$ are proper.  The Hamiltonians $p_j$ generate the
 $\R^d$-action
$$ \Phi_t = \exp t_1 H_{p_1} \circ \exp t_2 H_{p_2} \dots \circ \exp t_d
H_{p_d}.$$ We  denote  $\Phi_t$-orbits by   $\R^d \cdot (x, \xi)
$.

\begin{defn}\label{SING} We say: 
\bigskip

\begin{itemize}
\item A point $(x, \xi)$ is called a  regular point of ${\mathcal
P}$ if  $dp_{1} \wedge \cdot \cdot \cdot \wedge dp_{d}(x,\xi) \not=
0$. The (open) set of regular points in $S^*M$ is denoted by $S^*M_{\rm{reg}}$. \bigskip

\item A regular point is called `compact' if the orbit $\R^n \cdot (x, \xi)$ is compact. Otherwise it is non-compact.  The set of compact regular points, resp. non-compact
regular points,  in $S^*M$ is denoted by $S^*M_{\rm{reg},c}$ resp.  $S^*M_{\rm{reg},nc}$. \bigskip

\item A point $(x, \xi)$ is called a singular point of ${\mathcal
P}$ if  $dp_{1} \wedge \cdot \cdot \cdot \wedge dp_{d}(x,\xi) =
0$.
 An orbit $\R^n \cdot (x, \xi) $ of $\Phi_{t}$ is  singular if it is non-Lagrangean, i.e.   has dimension $<n$;
A level set  ${\mathcal P}^{-1}(c)$ of the moment map is called  a singular level if
it contains a singular point $(x, \xi) \in {\mathcal P}^{-1}(c)$.    (We then say $c$ is a singular value and write  $c \in B_{sing}$.)

\end{itemize}

\end{defn}
The orbits of the $\R^d$-action commuting with $G^t$ define a singular foliation of $S^* M$. 
We refer to \cite{El,El2, Zung, Zung96} for background on this foliation and its singularities. 
If  $b$ is  a regular value, we refer to  \begin{equation} \label{CI1}
{\mathcal P}^{-1}(b) = \Lambda^{(1)}(b) \cup \cdot \cdot \cdot
\cup
 \Lambda^{(m_{cl})}(b) , \;\;\;(b \in B_{reg})
\end{equation}
\noindent as a regular level set.   Here,    $m_{cl} (b) = \#
{\mathcal P}^{-1} (b)$ is  the number of  orbits on the level set
${\mathcal P}^{-1}(b)$; it is finite if $\pcal$ is real analytic. The Liouville-Arnold theorem (loc.cit.) states the following: \begin{theo}\label{LATH} Let $\pcal^{-1}(c)$ be a regular level
set of the moment map and let $X$ be a connected component. Then there exists  a neighborhood $D$ of $X$ in which the Hamiltonian $H(q,p)$ is symplectically equivalent to  
a Hamiltonian  $$H(x, y, I), \;\;\; (x, y, \theta, I) \in U \subset T^*\R^s \times T^* \T^{n-s} $$
which does not depend on $\theta$. The symplectic structure is $\sum_{j=1}^s dx^i \wedge d y^i
+ \sum_{i-1}^{n-s} d \theta_i \wedge d I_i$. \end{theo}
\noindent Note that $s =0 \iff X \subset (S^*M)_{reg, c}$.

\begin{lem} \label{OD} The set $(S^*M)_{reg, c}$ is open dense. \end{lem}

\begin{proof} By Sard's theorem, the set of regular points is open dense.  Further, $(S^*M)_{reg, c}$ is open as long as there exists
a single compact regular orbit.  It is only necessary to prove that non-compact regular orbits cannot fill out
an open set. But the closure of any orbit in that open set contains a singular locus invariant under the $\R^{n-1}$ action. That is, the closure contains a singular
orbit. That orbit lies on the same level set and so the level set of the orbit is singular. Hence an open set of non-compact singular orbits produces an open
set of singular levels, contradicting Sard's theorem.

\end{proof}

\subsection{\label{SINGPTSECT} Singular points}

At a singular point,  let
 $$K := \bigcap_{i = 1}^n\;
\ker d p_i (v),\;\;\;L = \mbox{span} \; \{H_{p_1}(v), \dots,
H_{p_d}(v)\}.$$
Then,
\begin{equation} \label{RANK}  \dim L := 
{\mbox rank}\, (dp_{1},...,dp_{d})|_{v} = k < d.
\end{equation}

As Lemma \ref{REGintro} indicates, singular points play an important role in this article. However, we do not need to know very
much about them.
Although we will not be using it, there is a  theorem due to Eliasson (and, independently, Vey) asserting that near a nondegenerate singular point, the associated singular
 Lagrangian foliation is diffeomorphic to that of the linearized system.  The relevant notion of  non-degeneracy  is unrelated to the Kolmogorov or
iso-energetic non-degeneracy conditions. Singular points in the case of ellipsoids are studied in \cite{Au, DDB, Zung}.

\subsection{Non-degeneracy, action-angle variables and frequency map at regular compact points}

 \begin{defn} \label{FREQMAPDEF} Assume $(x, \xi)$ is a compact regular point, so that a neighborhood of $(p, \xi)$ is foliated
by invariant tori $T_I$.
   Let $|\xi|_g^2 = H(I)$ be the expression of the Hamiltonian in local action-angle
variables in a neighborhood of $(x, \xi)$. Then, the frequency vector of the torus $T_I$ with action
variables $I$  is defined by $\omega_I = \nabla_I H$. The frequency map $ \omega_I: U\to \R^n$ is the map
$(x, \xi) \to \omega_{I(x, \xi)}. $  \end{defn}

 In the local action-angle variables, the  Hamilton orbit of the geodesic flow  has the form,
\begin{equation} \label{GEOAA} G^t(I, \phi) =   (I, \phi  + t \omega_I). \end{equation}

A type of `degeneracy' occurs when $\omega_I$ is `resonant'.

\begin{defn} A resonant torus is a torus $T_I$ for which the components $\omega_I$ are dependent
over $\Q$, i.e. there exists a non-zero integer vector $\vec n = (n_j) \in \Z^n$ so that $\langle \vec n, \vec \omega = \sum_j n_j \omega_j = 0. $  The resonant lattice of $\omega$ is the set of all $\vec n \in \Z^n$ such that $\vec n \cdot \vec \omega = 0$.  $\omega_I$ is said to be `non-resonant' if its components are independent over $\Q$.

 \end{defn}

\subsection{Non-compact regular orbits lie on singular levels}

\begin{lem} \label{NONCPTPROP} If compact regular orbits form an open dense set, then non-compact  regular orbits only occur on singular levels
of the moment map. \end{lem}

\begin{proof} Let $\R^{n-1} \cdot (x, \xi) = {\mathbb T}^k \times \R^{n-1-k} \cdot (x, \xi)$ be a non-compact regular orbit. The closure $\overline{\ocal}$  of the orbit  $
\ocal = \R^{n-1} \cdot (x, \xi) $
is a compact set lie on 
a level set of the moment map and  invariant under the $\R^{n-1}$ action.  The orbit  $\R^{n-1} \cdot (x, \xi) $ and its closure is open in the relative topology of 
the level set and of 
 $\overline{\ocal}$, and  $\overline{\ocal} \backslash \ocal$ consists of limit points of $R^{n-1-k} \cdot (x, \xi)$ and   is a relatively closed and compact subset.
 We claim that all points in the limit set are singular. If not, there exists a regular point $\zeta$ whose orbit is Lagrangian. But it is also compact. Hence, the orbit
 must be a Lagrangian torus. But then an open neighborhood of $\zeta$ has a Lagrangian torus orbit, contradicting the fact that  $R^{n-1-k} \cdot (x, \xi)$ 
 intersects the neighborhood.
 
\end{proof}

\subsection{The image of $\Lambda_{x_0}$}
Another general fact regarding self-focal points is the following general result. 
  \begin{lem} \label{IMLAMBDA} Let $(M,g)$ be any compact Riemannian manifold possessing a self-focal point $x_0$, and  $\Lambda_{x_0}$ be defined as in \eqref{FLDEF}. Let  
  $\pi: S^*\ecal_A \to \ecal_A$ be the natural projection. Then, $\pi(\Lambda_{x_0}) = \ecal_A$. 
  
  \end{lem}

\begin{proof}  It is well-known that the image $\exp_{x_0} T_{x_0} \ecal_A$ is all of $\ecal_A$, i.e. for any $y \in \ecal_A$, there
exists $\xi \in S^*_{x_0} \ecal_A$ and a time $t$ so that $y = \exp_{x_0} t \xi $. But $\pi (\Lambda_{x_0}) = \exp_{x_0} T_{x_0} \ecal_A$.

\end{proof}

\section{\label{MOSERSECT} Moser's isospectral approach to complete integrability of ellipsoids}

In this section, we review the  isospectral approach of Moser
 \cite{M80} to the complete integrability of the geodesic flow of an ellipsoid. We also review a nice observation of Audin on the interlacing property
 of eigenvalues of $A$ and eigenvalues of $L(x, \xi)$.

Let $A$ be a real positive definite symmetric $n \times n$ matrix ($A \in \rm{Sym}(n)$) with distinct eigenvalues (for the moment). An ellipsoid is
defined by 
$$\ecal_A: = \{x \in \R^n: \langle A^{-1} x, x \rangle = 1\}. $$
The equations of motion of a geodesic of the ellipsoid are given by curves on the ellipsoid satisfying,
$$\left\{ \begin{array}{l} \frac{d^2 x}{dt^2} = - \nu A^{-1} x, \\ \\
\nu = \frac{\langle A^{-1} y, y  \rangle}{|A^{-1} x|^2}, \;\;y = \frac{dx}{dt}. \end{array} \right. $$
Note that $\dim \ecal_A = n-1$, so that in addition to $H(x, \xi)= |\xi|_g$, complete integrability requires $n-2$ further Poisson commuting
functions.

Moser has a beautiful interpretation of the geodesic equations in terms of isospectral deformations of a Lax matrix,
\begin{equation} \label{MOSERLAX} L(x,y) = P_y (A - x \otimes x) P_y, \end{equation}
where $P_y$ is orthogonal projection onto the hyperplane orthogonal to $y$. Let $x \in \ecal_A$ and let $y = \frac{dx}{dt}$.
Then the eigenvalues of $L(x,y) $ are preserved under the geodesic flow. 

We will assume that $x \in \ecal_A$ and that $y \in S^*_x \ecal_A$. To emphasize these conditions on  $(x,y)$, we henceforth
use the notation $(x, \xi) \in S^* \ecal_A$. Thus, one obtains an $n \times n $ symmetric matrix for each $(x, \xi) \in S^*\ecal_A$. 
Two of the eigenvalues are $0$:

\begin{lem} \label{NORMALLEM} Let $x_0 \in \ecal_A$. Then,  $A^{-1} x$ is normal to $T_x \ecal_A$. Moreover,  $L(x_0, \xi)$ and $A - x_0 \otimes x_0^*$ annihilate the normal vector ${\bf n}$ at $x_0$. 
Hence, both take $T_{x_0} \ecal_A$ to itself. \end{lem}

\begin{proof} First we observe that $A^{-1} x$ is normal to $T_x \ecal_A$.
 Let  $x(t)$ be any curve on $\ecal_A$ with $x(0) = x, \dot{x}(0)= \xi$. Then,
 $\langle A^{-1} x(t), x(t) \rangle = 1$, hence,
 $$0 = \frac{d}{dt} |_{t =0} \langle  \langle A^{-1} x(t), x(t) \rangle  = 2 \langle A^{-1} x, \xi \rangle. $$ 
 Thus, ${\bf n} = \frac{A^{-1} x_0}{|| A^{-1} x_0||}$. 
Then,
$$(A - x_0 \otimes x_0^*) {\bf n} = 0 \iff (A - x_0\otimes x_0^*) A^{-1} x_0 = 0 \iff x_0 = \langle A^{-1} x_0, x_0 \rangle x_0 = 0.$$
On the other hand, 
$$L(x_0, \xi) {\bf n} = P_{\xi} (A - x_0 \otimes x_0^*) P_{\xi} {\bf n} = P_{\xi} (A - x_0 \otimes x_0^*) {\bf n} =0,$$
since $\xi \in S^*_{x_0} \ecal_A$ is orthogonal to ${\bf n}$. \end{proof}

\begin{rem} Moser does not point this out explicitly in \cite{M80} because $x$ is allowed  to range over all of $\R^n$.

\end{rem}

Moreover, by definition $L(x, \xi) \xi =0$. Thus, 

\begin{cor}\label{Mult2LEM}  For all $(x, \xi) \in S^*\ecal_A$, $0$ is an eigenvalue of $L(x, \xi)$ of multiplicity two. The corresponding eigenvectors 
are ${\bf n}_x$ and $\xi$. \end{cor}

 We further observe the following:
 \begin{lem}
 \label{Only2LEM}  If $\xi \in S^*_x \ecal_A$ and $\eta \bot \xi$, $\eta \not=0 $, then  $L(x, \xi) \eta \not= 0.$
  \end{lem}
 
 \begin{proof} We may assume for purposes of contradiction that $\eta \bot \xi$, so that $P_{\xi} \eta = \eta. $
 Then $L(x, \xi) \eta =0$ implies that $P_{\xi} (A - x \otimes x) \eta = 0$, i.e that $(A - x\otimes x) \eta = c \xi.$
 But, $\langle \eta, \xi \rangle = 0$ and $\langle A^{-1} x, \xi \rangle = 0$ implies 
 $$\begin{array}{lll} A \eta - \langle x, \eta \rangle x = c \xi  & \implies &  \eta - \langle x, \eta \rangle A^{-1} x = c A^{-1} \xi \\&&\\
 & \implies & \langle \xi,  \eta \rangle  - \langle x, \eta \rangle \langle A^{-1} x, \xi \rangle = c \langle  A^{-1} \xi, \xi \rangle
 \\&&\\
 & \implies &  0  = c \langle  A^{-1} \xi, \xi \rangle, 
 
 \end{array}$$ 
 a contradiction.

 \end{proof}

 \subsection{Eigenvalue moment map}
 
We denote the non-zero eigenvalues by $\lambda_1(x, \xi) \leq  \lambda_2(x, \xi) \leq  \cdots \leq \lambda_{n-2}(x, \xi). $ Since $\lambda_j(x, \xi)$ become
singular on the {\it coincidence set} where two eigenvalues are equal, 
 it is better to use a basis of invariant polynomials, such as the
 elementary symmetric functions $e_j(\vec \lambda)$  or the power functions $p_k(\vec \lambda) = \sum_{j=1}^{n-2} \lambda_j^k$. However,
 as explained below, these are not coordinates and their differentials vanish at $(x, \xi)$ where $L(x, \xi)$ have multiple eigenvalues.
 In other words $d e_j$ are not everywhere independent on $\R^{n-2}$.
On the vector space $\R^{n-1}$ with coordinates $(\lambda_1, \dots, \lambda_{n-1})$, 
 $d e_1 \wedge \cdots \wedge d e_{n-1} = \Delta(\vec \lambda) d \lambda_1 \wedge \cdots \wedge d \lambda_{n-1}$. The right side vanishes
on the `coincidence set' $\dcal$  where some pair $\lambda_i = \lambda_j$ with $i \not=j$.

 We pick one of these bases of symmetric polynomials, and obtain a moment map,
 $$\pcal: T^*\ecal_A \to \R^{n-2}, \;\; \pcal(x, \xi) = (e_1 (\vec \lambda), \dots, e_{n-2}(\vec \lambda)), \vec \lambda = \rm{Sp}(L(x, \xi)). $$
 Since $L(x, \xi)$ is homogeneous of degree zero in $\xi$ we restrict it to $S^* \ecal_A$ and define,
  \begin{equation} \label{pcaldef} \pcal: S^*\ecal_A \to \R^{n-2}, \;\; \pcal(x, \xi) = (e_1 (\vec \lambda), \dots, e_{n-2}(\vec \lambda)), \vec \lambda = \rm{Sp}(L(x, \xi)). \end{equation}  
The set of points where $\pcal$ achieves its maximal rank $n-2$
is  open dense (see Corollary \ref{PCALFULLRANK}).   Since   $\dim : S^* \ecal_{A}  = 2n - 3$,  $\pcal : S^* \ecal_{A} \to \R^{n-2}$ has rank $n-2$ and the generic inverse image under the moment map
has dimension $2n-3 - (n-2) = n -1 = \dim \ecal_A$, i.e.  is a Lagrangian submanifold
 of $T^* \ecal_A$ (lying in $S^* \ecal_A$), and  $\pcal |_{S^*\ecal_A} \to \R^{n-2}$ is a (singular) Lagrangian torus fibration over its image.  
As above, the $e_j(\vec \lambda)$ (the elementary symmetric functions of the eigenvalues of $L(x, \xi)$) are not coordinates
on $S^* \ecal_A$. Indeed,
 $\pcal^*   d e_1 \wedge \cdots \wedge d e_{n-2} $
vanishes whenever $L(x, \xi)$ has a double (non-zero) eigenvalue. 
By definition, a level set of the eigenvalue moment map $\pcal$ in $S^* \ecal_A$ is the  isospectral set by  $$\mcal(\vec \lambda): = \{L(x,y) : \rm{Sp}(L(x,y)) = \{\lambda_1, \dots, \lambda_{n-2}, 0,0\}. $$


For the record, we note that Moser defines (\cite[Page 150]{M80}) for $ (x,y) \in S^*\ecal_A$,  the skew-symmetric  
$n \times n$ matrix 
\begin{equation} \label{BDEF} B(x,y) = - \begin{pmatrix} \alpha_i^{-1} \alpha_j^{-1} (x_i y_j - x_j y_i)\end{pmatrix}, \end{equation}   and
shows that the Lax equations are, \begin{equation} \label{LAXEQ} \dot{L} = [B, L], \;\;\; L_t(x, \xi) = L(G^t(x, \xi))= e^{t B(x, \xi)} L(x, \xi)  e^{ - t B(x, \xi)}.\end{equation} 

\subsection{Confocal quadrics, eigenvalues and ellipsoidal coordinates}

\begin{defn} The quadrics confocal with $\ecal_A$ are defined by   \begin{equation} \label{QUADRICDEF} {\mathfrak A}_z: = \{x: \langle (A - z)^{-1} x, x \rangle = -1, \;\; z \in \R. \}  \end{equation}
Here, ${\mathfrak A}_0 = \ecal_A$.
\end{defn}

According to \cite[Page 166]{M80}, \begin{lem} The eigenvalues $\lambda_1, \dots, \lambda_{n-2}$ are the values of $z$ so that the line from $x$ with direction $\xi$ 
intersects the quadric \eqref{QUADRICDEF} tangentially. 
  \end{lem}
 Note that Moser includes $\ecal_A$ as one of the quadrics, and therefore writes of $n-1$ confocal quadrics. 
 
 We briefly review the proofs to establish notation that we use later on. Define
 \begin{equation} \label{BigPhi} \Phi_z(x, \xi): = \frac{|\xi|^2}{z} \frac{\det (L(x, \xi) - zI)}{\det (A -z I)}. \end{equation}
 It is a rational function with $n-1$ zeros at $\lambda_1(x, \xi), \dots, \lambda_{n-2}(x, \xi), \lambda_{n-1} = 0$ and with poles at the eigenvalues $\alpha_j$
 of $A$. Hence, the equation for the eigenvalues of $L(x, \xi)$ is,  \begin{equation} \label{EIGEQ} \Phi_z(x, \xi)  = 0, \end{equation}
 whose solutions are the quadratic  cone of tangents to ${\mathfrak A}_z$ through $x$.  One has that $\Phi_z(x, \xi) = 0 \implies  \Phi_z(x + s \xi, \xi) = 0$ for all $s$.
 Let,
 $$Q_z(x, y) = \langle (z - A )^{-1} x, y \rangle, \;\; Q_z(x) = Q_z(x,x). $$
 Note that $x + t \xi$ is tangent to $Q_z$ if and only if the quadratic form $Q_z(x + t \xi)$ has a double root.
 An important identity is that \cite[P. 162]{M80}, 
 \begin{equation} \label{IDENTITY} \Phi_z(x, \xi) = Q_z (\xi) (1 + Q_z(x)) - Q_z^2(x, \xi). \end{equation}

 \subsection{Eigenvectors of $L(x, \xi)$}

When the $\lambda_j$ are distinct, there is a unique (up to signs) orthonormal basis of eigenvectors $\phi_j$. When the eigenvalues are
multiple, we pick an orthonormal basis for each eigenspace. According to \cite[Page 166]{M80}, we have,
\begin{lem} \label{ONBLEM} 
  If the eigenvalues  $\{\lambda_j\}_{j=1}^{n-2}$ of $L(x, \xi)$  are distinct, then
the eigenvectors $\phi_j(x,\xi)$ are the normals to the confocal quadrics ${\mathfrak A}_{\lambda_j}$ at the point of contact of the line
through $x$ with direction $\xi$. Moreover, $ \{\phi_j(x_0, \xi) \}_{j=1}^{n-2}$ is an orthonormal basis of $T_{\xi}  S^*_{x_0} \ecal_A $ and $\{{\bf n}, \xi, 
\phi_j(x_0, \xi)\}$ is an orthonormal basis of $\R^n$. \end{lem}

\begin{proof} 
We only prove the last statement since the first is proved in \cite{M80}.
Since $\dim S^*_{x_0} \ecal_A = n-2$,  $\xi \bot T_{\xi}  S^*_{x_0} \ecal_A, $  $\phi_j(x, \xi) \bot \xi$, $\phi_j(x, \xi) \bot {\bf n}$, and $\{\phi_j\}$ is an orthonormal
set, it must be an orthonormal basis of $T_{\xi}  S^*_{x_0} \ecal_A $. \end{proof}

\subsection{\label{AUDINSECT} Some results of Audin}
In this section we review some useful results stated in \cite{Au} relating eigenvalues of $A$ and ellipsoidal coordinates. 

Define $\rcal(z; x_1, \dots, x_{n})$ by,  $$\sum_{j=1}^n \frac{x_j^2}{\alpha_j -z} = \frac{\rcal(z; x_1, \dots, x_{n})}{\prod (\alpha_j -z)}, \;\; \rm{i.e.} \;\;\rcal(z; x_1, \dots, x_{n})  := \sum_{j=1}^{n} x_j^2 \prod_{j \not= i} (\alpha_i - z). $$
The roots of $\rcal(z; x_1, \dots, x_n)$ are denoted by $\zeta_j$ in \cite[Lemma 2.6.1]{Au}, and are the Jacobi ellipsoidal coordinates of $x$. The following is  \cite[Lemma 2.6.1]{Au} (see also  \cite[Page 3]{DDB}.

 \begin{lem} \label{JACINT}  If the eigenvalues of $A$ are distinct, then 
the  eigenvalues of $A$ and the ellipsoidal coordinates   interlace:
$$\zeta_0 \leq \alpha_0 \leq \zeta_1 \leq \alpha_1 \leq \zeta_2 \leq \alpha_2 \cdots \leq \zeta_n \leq \alpha_n. $$
Moreover, $\zeta_j = a_j$ only if $x_j = 0$. \end{lem}

\begin{cor}\label{lambdaalpha}  If $\zeta_j$ are enumerated in increasing order, then $\zeta_j = \zeta_{j+1}$ if and only if
 $\zeta_j = \zeta_{j+1} = \alpha_j.$ Also,  
 if  $A$ has a double eigenvalue, say $\alpha_j = \alpha_{j+1}$. Then $\zeta_j = \alpha_j$.
 \end{cor}

 The following important observation is stated in \cite[Proposition 2.6.2]{Au}  \cite[Page 195 (2)]{Au} and 
 Define the polynomial $\pcal(z)$ by 
$$\begin{array}{lll} \Phi_z(x, \xi) & = &  -  \frac{\pcal(z)}{\prod (\alpha_j -z)}. \end{array}$$

\begin{lem} \label{AUDINLEM} The polynomial $\pcal$ (hence $\Phi_z(x, \xi)$)  has exactly one  real root in each interval $(-\infty, \zeta_1), [\zeta_1, \zeta_2], \dots, [\zeta_{n-1}, \zeta_n]$. All of its roots
are real. 
\end{lem}

 \section{\label{ISOESECT} Iso-energetic non-degeneracy of an ellipsoid geodesic flow}

 In this section, we prove Proposition \ref{MAINNDintro}. We begin by reviewing the definitions.
 
 On the open dense subset of regular compact torus orbits, there exist local action-angle variables $(I_1, \dots, I_h, \theta_1, \dots,
\theta_n)$. This defines a second moment map $\ical = (I_1, \dots, I_n)$. The $p_j$ are functions of $\ical$ and the two sets of
generators define the same foliation. We write $H(x, \xi) = F(I_1, \dots I_n)$,.

\begin{defn} \label{KNDDEF} The frequency map of Definition \ref{FREQMAPDEF}  is Kolmogorov {\em non-degenerate} if $$\det \begin{pmatrix} \frac{\partial \omega_j}{\partial I_k} \end{pmatrix}
= \det \begin{pmatrix} \frac{\partial^2 H}{\partial I_j \partial I_k } \end{pmatrix} \not= 0. $$
If an  energy level is fixed,   the Hamiltonian is said to be iso-energetically non-degenerate if
\begin{equation} \label{DET}  \det \begin{pmatrix} \frac{\partial^2 H}{\partial I_j \partial I_k } & \frac{\partial H}{\partial I_j  } \\&&\\
\frac{\partial H}{\partial I_k }  & 0\end{pmatrix} \not= 0. \end{equation}
\end{defn}
Non-degeneracy of the frequency map is the condition that $D \omega: U \subset M \to \R^d $
be surjective. 
Iso-energetic non-degeneracy is the condition  that $$\omega: U \subset \{H = E\} \to \R{\mathbb P}^n, \;\; \zeta \to [\omega_1: \omega_2: \cdots :\omega_n] $$
is an immersion.

The two notions of non-degeneracy are related by the Schur determinant formula, as follows.

\begin{prop} \label{ISOENDPROP} Suppose that $H: T^* M \to \R$ and let  $\{H = E\}$ be an energy surface. Let $\vec I$ be local action variables in a neighborhood $U$ of
a point on $\{H = E\}$. 
Then, $H$ is iso-energetically non-degenerate $U \cap \{H = E\}$ if  (i) it is  Kolmogorov non-degenerate in $U$ and (ii) if the frequency map $\omega = \nabla_I$
is non-vanishing in $U$. 

\end{prop}

\begin{proof}


We recall that the Schur complement of 
$A$ in the matrix  $\begin{pmatrix} A &  B \\ & \\ B^* & D \end{pmatrix}$ is given by
$D - B^* A^{-1} B$, and that
$$\det A \det (D - B^* A^{-1} B). $$
Since $D = 0$ in \eqref{DET}, the iso-energetic determinant equals
\begin{equation} \label{SCHUR}  \det (D^2_{I_j, I_k} H) \langle (D^2_{I_j, I_k})^{-1} \nabla H, \nabla H \rangle. \end{equation}
If the system is Kolmogorov non-degenerate, then $\det (D^2_{I_j, I_k} H)  \not=0$. To prove that it is iso-energetically non-degenerate when
it is known to be Kolmogorov non-degenerate, it suffices to prove that $  \langle (D^2_{I_j, I_k})^{-1} \nabla H, \nabla H \rangle \not= 0$. Since
$D^2_{I_j, I_k})^{-1}$, iso-energetic  non-degenercy holds as long as   $\nabla_I H \not= 0$.

\end{proof}


To illustrate the two notions, 
we consider examples of non-homogeneous Hamiltonians where the two notions of non-degeneracy differ along a low-dimensional set. Let
$n=2$, and consider the Hamiltonians  $H_1, H_2: T^*\T^2 \to \R$ defined by, $H_1(x,\xi) = \xi_1 + \xi_2 + \xi_1^2$ and $H_2(x, \xi) = \xi_1 + \xi_2 + \xi_1^2-\xi_2^2$.
Then at the torus $\T^2 \times \{0\}$, $X_{H_1}$ is isoenergetically non-degenerate but Kolmogorov degenerate, and $X_{H_2}$ is the 
reverse. See also \cite[Page 13]{Roy}.

\subsection{\label{KNDISONDSECT} Proof of Proposition \ref{MAINNDintro}:  Kolmogorov non-degeneracy plus homogeneity implies iso-energetic non-degeneracy}


The purpose of this section is to prove Proposition \ref{MAINNDintro}.

 In \cite[Theorem 6.16]{K85}, H. Kn\"orrer proved that the frequency map of an ellipsoid with distinct axes is non-degenerate almost everywhere.
We will also use the following result \cite[Proposition 5.2.9]{Kl95} (see also \cite{Sch72}):
\begin{prop} The geodesic flow of an ellipsoid with pairwise different principal axes is completely integrable with (Kolmogorov) non-degenerate period mapping. In
particular, the periodic orbits of the geodesic flow are dense. \end{prop}
Combining these results with Proposition \ref{MAINNDintro} gives,

 \begin{cor}   \label{ISOENDPROP2} Suppose that $\ecal_A$ is a a multi-axial  ellipsoid. Then the geodesic flow of $\ecal_A$ is iso-energetically 
 non-degenerate on a dense open set. \end{cor}
 
 We now prove Proposition \ref{MAINNDintro}.
The main idea is to study the image of $S^*M$ and $T^*M$ under the frequency map $$(x, \xi) \to \nabla_I F(I(x, \xi))=:  \omega_I(x, \xi)$$ from $T^*M \backslash 0 \to \bcal^*$. 
The proof proceeds by a sequence of Lemmas.  First, we study the homogeneities of moment maps and frequency map.
Of course, the Hamiltonians $H(x, \xi)= |\xi|_g^2$ resp. $|\xi|_g$ are homogeneous. The geodesic flow is homogeneous.
It is not apriori clear whether the actions $p_j$ or $I_j$ are homogeneous. However, the dilates of the integrals $p_j$ commute
with $H$ since the Poisson bracket (symplectic form) is homogeneous and since $H$ is homogeneous.

\begin{lem} The torus fibration obtained by orbits of $\pcal$
or $\ical$ is independent of paramerization and is also invariant under dilations, $(x, \xi) \to (x, r \xi)$. \end{lem}

This statement is not obvious because we do not have explicit formulae for the constants of the motion.

\begin{proof} The statement is obviously true for the geodesic flow-lines foliation, since the Hamiltonian is homogeneous. However,
for a generic Lagrangian torus, the geodesic flowlines are dense in the torus. Since each is invariant under dilation, so must be their closure,
hence the torus must be invariant under dilation.  Hence the statement is true for an open dense set of orbits. But then it must be true for all
orbits by continuity of the flow.

\end{proof}

\begin{lem} For fiber-homogeneous Hamiltonians,  action-angle variables are homogeneous of degree $1$  under dilation $\tau(x, \xi) = (x, \tau \xi)$  of tori in $T^* E_e$. \end{lem}

\begin{proof} For regular tori, action variables are defined by
 $$I_j ((x, \xi) ) = \oint_{\gamma_j((x, \xi) )} pdq, $$
 where $\{\gamma_j(x, \xi)\}$ is a homology basis for the torus $\R^n \cdot (x, \xi)$. Since $pdq$ is homogeneous of degree one, 
 $$I_j (x, \tau \xi) =  \oint_{\gamma_j((x, \tau \xi) )} pdq=  \oint_{\gamma_j((x, \xi) )} \tau^* pdq = \tau  \oint_{\gamma_j((x, \xi) )} pdq. $$ \end{proof}

\begin{lem} If $H(x, \xi)$   is homogeneous of degree one, resp. two,  under dilation $\tau(x, \xi) = (x, \tau \xi)$, and if $H = F(I)$ is its expression
in action variables, then $F$ is homogeneous of degree one, resp. two,  on $\bcal$. \end{lem}

\begin{proof} We have, 
$$\tau F(I(x, \xi)) = \tau H(x, \xi) = H(x, \tau \xi) = F(I(x, \tau \xi)) = F(\tau I(x, \xi)) \implies F(\tau b) = \tau F(b).  $$

\end{proof}

\begin{lem} \label{omega0} If $H(x, \xi)$   is homogeneous of degree two, resp. one,  under dilation, then  $\omega_I$ is homogeneous of degree one, resp. zero, under dilation $\tau(x, \xi) = (x, \tau \xi)$  of tori in $T^* E_e$. \end{lem}

\begin{proof}  Since $\omega_I (x, \xi) = \nabla_I F(I(x, \xi)) $ it suffiices to show that $\nabla_I F(I)$ is homogeneous of degree one, resp. zero
on $\bcal$. But $F(I)$ is homogeneous of degree two, resp. one. 

\end{proof}

For instance, $I_j(x, \xi) = \xi_j$ in the case of a flat torus and $F(I) = \sum_j I_j^2/2 = H(x, \xi) $. Here we normalize $H$ to be homogeneous
of degree two, and  $\omega_I = (I_1, \dots, I_n)$. If we define $H = |\xi|_g,$ then $F(I) = \sqrt{\sum_j I_j^2/2}  $ and
$\omega_I = \frac{(I_1, \dots, I_n)}{\sqrt{\sum_j I_j^2/2} }. $

Next, we study the geometry of the frequency map $f = \omega_I$. Some of the following overlaps the viewpoint of  \cite{BH91}.
The frequency map is viewed as a map $f: \vec I \to \omega_{\vec I}$  from $\R^n \to \R^n$, or in Duistermaat's notation,  $f: \vec I \in \bcal \to
\R^n$.  Here, $H = F(I)$ and $H^{-1}(1) = S^*M$ and  $(f \circ \pi) (H^{-1}(0))$
is the image of $\{H = 0\}$ in $T^* \T^n$ under the frequency map rather then the action moment map. So $(f \circ \pi) (H^{-1}(1))$ is the image of $S^*M$ under the frequency map.

\begin{lem}   Iso-energetic non-degeneracy is equivalent to: the 
half lines $\R_+ \vec \omega$ should be transverse to the hypersurface $(f \circ \pi)(H^{-1}(1))$, where $\pi: T^*\T^n \to \R^n$ is the projection
onto the momentum axis.  \end{lem}

The statement is
that it is transverse to all rays through the origin.
Compare the two statements: (i) $(F'')^{-1} \omega_I$ is transverse to $\{F = c\}$, and (ii)  $\R_+ \vec \omega$ should be transverse to the hypersurface $
\omega (S^*M)$. In the second, $\R_+ \vec \omega$ is simply any ray through the origin.

\begin{proof}
The map $\vec \omega \circ \ical^{-1}$ converts $\ical(S^*M) \subset \bcal$ into $\omega_I(S^*M)$. Its inverse is $\ical \circ \vec \omega^{-1} $. Now,
$\vec \omega = \nabla_I F$, so $D \vec \omega = F''$ and $(F'')^{-1}  = D \vec \omega^{-1}$.

Let $\Omega = \nabla_I F \circ \ical (T^*M)$.
The basic maps  are  $$(x, \xi) \in T^* M \to \vec I(x, \xi) \in \bcal \to \nabla_I F \in \Omega. $$
Restrict to the cosphere bundle to get
$$(x, \xi) \in S^* M \to \vec I(x, \xi) \in \bcal_1 = \{F = 1\}  \to d_I F (\bcal_1) \in \Omega \subset T^*\bcal. $$
Then, $D F' : T_I \bcal \to T_{\omega_I} \Omega$. 

In effect, $I \to \omega_I$ is a kind of Gauss map taking a point to its normal.
$D^2 F = D \omega_I$ takes the tangent space at $I$ to $\{F =1\}$ to the tangent space at $\omega_I$  of its Gauss map  image $\omega\{F =1\}$.
Now, $\omega_I$ is normal at $I$ to $\{F =1\}$, so $D \omega_I$ takes $\omega_I$ (viewed as a tangent vector to $\bcal$) to
a vector $(D_I  \omega_I) \in T_{\omega_I} \omega (\bcal)$. The derivative takes tangent vectors to $\{F=1\}$ to tangent vectors
to $\omega \{F =1\}$, so $D_I \omega_I \cdot \omega_I$ is  transverse to $\omega_I \{F =1\}$ at $\omega_I$ when $D \omega_I$ is Kolmogorov non-degenerate. 
The ray $\R_+\omega_I$ hits $\omega_I\{F=1\}$ at $\omega_I$ and so this ray is transverse to $\omega \{F =1\}$.

\end{proof}

\begin{lem} \label{CONE} If $H$ is real analytic and homogeneous of degree 2, and if it fails to be iso-energetically non-degenerate, then $S: = \omega_I(S^*M)$
is the cone over a link $L \subset S^{n-1}$ in $\bcal^*$. \end{lem}

\begin{proof} Isoenergetic degeneracy at one point $\omega_I$ is the condition that the ray $\R_+ \omega_I$ is tangent to $S$ at $\omega_I$.
If iso-energetic non-degeneracy fails to hold on an open set, the by real-analyticity it fails to hold on a dense open set. It follows that
for a dense open set of $\omega_I \in S$,  $\R_+ \omega_I$ is tangent to $S$ at $\omega_I$. This is a closed condition, and therefore it holds
everywhere on $S$.


We now need to show that,  if $S$ is a real analytic  hypersurface such that every ray through $0$ is tangent to $S$ at its intersection point, then $S$ is open subset of a 
cone over a link $L \subset S^{n-1}$ of codimension $1$ in $S^{n-1}$ i.e. $S \subset \R_+ L$.

Note that $S = \bigcup_{\eta \in S^{n-1}} \R_+ \eta \cap S. $ Indeed, through every point $v$ of $S$, the ray $\R_+ v$ obviously hits $S$. 
Define $S/\sim$ to be the quotient of $S$ by the equivalence relation of belonging to the same ray. 
We then have a natural  map $p: S \to S^{n-1}$ by following $\R_+ x$ for $x \in S$ to its intersection with $S^{n-1}$.  We wish to show that
the image is of codimension one, namely a link. Since the is analytic, the alternative  is that the image contains an open disk $D^{n-1} \subset S^{n -1}$. 
Then the disk and $S$ have the same direction, and follwing the rays gives a map $g: D^{n-1} \to S$, so that the image is a graph over
the disk. But the one shows that the rays cannot all be tangent at the impact point unless the image is planar,  since $g_* x$ will cover the tangent spaces at impact
points.


\end{proof}


We now complete the proof of Proposition \ref{MAINNDintro}.

Let us normalize $H$ to be homogeneous of degree two and  consider the image $\omega_I(\{F =1\} = \omega_I(S^*M)$. By Lemma \ref{omega0}, the image $\omega_I(T^* M)$ is the cone over $ \omega_I(S^*M)$.  But by Lemma \ref{CONE}, $\omega_I(\{F =1\} $ is itself a cone over a
link  $L \subset S^{n-1}$ and at the point $\omega_I $ in the image, the ray $\R_+ \omega_I$ intersects $ \omega_I(S^*M)$ in an interval (possibly
infinite). If we dilate in the $I$ variables, the by Lemma \ref{omega0}, we  dilate the image in the $\omega_I$ variables, and the dilation $\omega_I$
along the same ray $\R_+ \omega_I$. It follows that the final dilation variable does not give an independent direction to the image, and therefore
the integrable system cannot be Kolmogorov non-degenerate.

\section{Variation of eigenvalues of $L$}
\label{LMULTSECT} 
In this section, we compute the variation of the eigenvalues $\lambda_j(x,\xi)$ of $L(x,\xi)$ along real analytic curves in $\SE_A$. The resulting equation \eqref{FIRSTVAR} allows us to prove that there is an open dense set $\mathcal{U}\subseteq \SE_A$ on which $L$ has simple nonzero eigenvalues, and $\calP$ is of full rank.

Let $(x(t),\xi(t))$ be an analytic curve in $S^*\SE_A$. The condition that $(x,\xi)\in S^*\SE_A$ can be formulated as
\begin{eqnarray*}
	\ang{A^{-1}x,x}&=&1\\
	\ang{A^{-1}x,\xi}&=&0\\
	\ang{\xi,\xi}&=&1.
\end{eqnarray*}
Differentiating in $t$, we may consider variations $(\dot x,\dot \xi)\in T_{(x,\xi)}S^*\SE_A$ with
\begin{eqnarray}
\nonumber\ang{A^{-1}x,\dot x}&=&0\\
\label{eq:validvar}\ang{A^{-1}x,\dot \xi}+\ang{A^{-1}\dot x,\xi}&=&0\\
\nonumber \ang{\dot \xi,\xi}&=&0.
\end{eqnarray} 
Along any choice of curve $(x(t),\xi(t))$ we can continue eigenvalues analytically, and we denote these analytic branches as $\lambda_j(t)$ with $\phi_j(t)$ a corresponding orthonormal basis of eigenvectors. Now, writing $B=A-x\otimes x$ for brevity, we compute the variation of a particular eigenvalue.
\begin{eqnarray*}
	\dot\lambda&=& \ang{\dot L \phi ,\phi}\\
	&=& \ang{\dot PBP\phi,\phi}+ \ang{P\dot BP\phi,\phi}+ \ang{ PB\dot P\phi,\phi}\\
	&=& \ang{\dot PB\phi,\phi}+ \ang{\dot B\phi,\phi}+ \ang{B\dot P\phi,\phi}\\
	&=& -\ang{(\dot \xi \otimes \xi+\xi\otimes \dot \xi)B\phi,\phi}+ \ang{\dot B\phi,\phi}- \ang{B(\dot \xi \otimes \xi +\xi \otimes \dot \xi)\phi,\phi}\\
	&=&-\ang{(\dot \xi \otimes \xi)B\phi,\phi}+ \ang{\dot B\phi,\phi}- \ang{B(\xi \otimes \dot \xi)\phi,\phi}\\
	&=&-2\ang{\dot\xi,\phi}\ang{B\phi,\xi}-\ang{(\dot x\otimes x+ x\otimes \dot x)\phi,\phi}\\
	&=& -2\left(\ang{\dot\xi,\phi}\ang{B\phi,\xi}+\ang{\dot x,\phi}\ang{x,\phi}\right)
\end{eqnarray*}
Hence we have
\begin{equation}
\label{FIRSTVARFULL}
\dot\lambda=-2\left(\ang{\dot\xi,\phi}\ang{(A-x\otimes x)\phi,\xi}+\ang{\dot x,\phi}\ang{x,\phi}\right).
\end{equation}
In the case $\dot x=0$ of vertical variations, this reduces to 
\begin{equation}
\label{FIRSTVAR}
\dot\lambda=-2\ang{\dot\xi,\phi}\ang{(A-x\otimes x)\phi,\xi}.
\end{equation}

\begin{prop}
	\label{prop:strongmult}
	The matrix $L(x,\xi)$ has eigenvalues $0=0<\lambda_1(x,\xi),\ldots \leq \lambda_{n-2}(x,\xi)$ counting multiplicity. On an open dense set $\calU\subseteq \S^*\SE_A$, the nonzero eigenvalues $\lambda_j(x,\xi)$ are distinct, and the map $(x,\xi)\mapsto (\lambda_1,\ldots,\lambda_{n-2})$ is a smooth map of full rank.
\end{prop}

\begin{proof}
The first part of this Proposition follows from Lemma \ref{Only2LEM}. Next we show that the nonzero eigenvalues are distinct on an open dense set. If this were not true, then there would exist an open set $\mathcal{B}\subset S^*\SE_A$ such that some nonzero eigenvalue $\lambda(x,\xi)$ has multiplicity $k\geq 2$. Fix $(x_0,\xi_0)\in\mathcal{B}$ and $V=\Span \{\phi_j\}_{j=1}^k \subset \RR^n$ the $\lambda$-eigenspace at $(x_0,\xi_0)$. We now choose a curve $(x(t),\xi(t))$ through $S^*\SE_A$ so that $\dot x=\phi_i\in (A^{-1}x)^\perp$ and 
\[\dot\xi=-\frac{\ang{A^{-1}\xi,\phi_i}}{|A^{-1}x|^2} A^{-1}x+c\phi_i\]
with $c\in\RR$ arbitrary. This is always possible, as this choice of $\dot x,\dot \xi$ satisfy \eqref{eq:validvar}. 

Now applying \eqref{FIRSTVARFULL}, we obtain $\dot\lambda_j=0$ for $j\neq i$ and
\begin{equation}
\dot\lambda_i=-2(\ang{\phi_i,x}+c\ang{(A-x\otimes x)\phi_i,\xi})
\end{equation}
From our assumption that the $\lambda_j$ cannot be split by any of these variations, it follows that for every corresponding eigenvector $v\in\ker(L-\lambda I)$ and every $c\in\RR$, we have $\dot\lambda=0$ and hence $\ang{v,x+c(A-x\otimes x)\xi}=0$. Hence $\ker(L-\lambda I)\subseteq x^\perp \cap ((A-x\otimes x)\xi)^\perp$, and for any eigenvector $v$ we have $\lambda v=Lv=P(A-x\otimes x)Pv=P(A-x\otimes x)v=(A-x\otimes x)v=Av$. This is a contradiction, as the eigenspaces of $A$ are $1$-dimensional.

Hence we have an open dense set $\calU\subset S^*\SE_A$ on which the eigenvalues $\lambda_j$ are distinct. On $\mathcal{U}$, these $\lambda_j$ are analytic functions. 
Finally the eigenvalue map  $(x,\xi)\mapsto (\lambda_1,\ldots,\lambda_{n-2})$ were not of full rank at some $(x_0,\xi_0)\in\calU$, we would have a linear dependence  $\sum_{j=1}^{n-2} c_jd\lambda_j(x_0,\xi_0)=0$. Pairing this form with the same variation $(\dot x,\dot \xi)$ considered above, we obtain $c_i=0$. As this is true for every $i$, indeed the differentials of the $\lambda_j$ are linearly independent on $\calU$. 
\end{proof}

\begin{cor}
\label{PCALFULLRANK}
	The moment map $\pcal$ introduced in \eqref{pcaldef} is of full rank in an open dense subset $\calU\subseteq S^*\SE_A$.
\end{cor}
\begin{proof}
the  elementary symmetric functions in the nonzero eigenvalues $\lambda_j$ are given by $$e_j (\vec \lambda)=
\sum_{1 \leq p_1 < \cdots < p_j \leq n-2} \lambda_{p_1} \cdots \lambda_{p_j}.$$
It is then well known that
$$d_{\vec \lambda}  e_1(\vec \lambda  ) \wedge \ldots \wedge d_{\vec \lambda} e_{n-2} (\vec \lambda) = \Delta(\vec \lambda) \prod_{j=1}^{n-2}
d \lambda_j, $$
where  $ \Delta(\lambda_1, \dots,
\lambda_{n-2}) = \prod_{1 \leq j < k \leq n-2} (\lambda_k - \lambda_j)$ is the
Vandermonde determinant. It follows that
\begin{equation} \label{DET0} d_{x,\xi}  e_1(\vec \lambda (x_0, \xi_0) ) \wedge \cdots \wedge d_{x,\xi} e_{n-2} (\vec \lambda (x_0, \xi_0)) =
\Delta(\vec \lambda) \prod_{j=1}^{n-2}
d_{x,\xi}  \lambda_j. \end{equation}

Noting that $\Delta(\vec \lambda) =0$ precisely on the coincidence set $\mathcal{D}\subset \calU^c$ where at least two eigenvalues coincide, the maximality of rank for $\mathcal{P}:(x,\xi)\mapsto (e_1,\ldots,e_{n-2})$ follows from Proposition \ref{prop:strongmult}.
\end{proof}

\section{Proof of Proposition \ref{PSFSING}}

The purpose of this section is to prove:
\begin{prop} \label{PSFSING} 
 Suppose that $x_0$ is a self-focal point of an ellipsoid $\ecal_A$ with pairwise different principal axes. Then for
all $\xi \in S^*_{x_0} \ecal_A$,  $(x_0, \xi)$ is either singular or has a non-compact orbit. In pariticular, $(x_0, \xi)$ lies on a singular component of the moment map.  \end{prop}

There are two possibilities: 
\begin{itemize}
\item $(x, \xi)$ is a singular point of $\pcal$, and its orbit is of dimension $< n-1$ (See Definition \ref{SING} and Section
\ref{SINGPTSECT}); \bigskip

\item $(x, \xi)$ is a regular point of $\pcal$ but its $\R^{n-1}$ orbit is non-compact of type ${\mathbb T}^k \times \R^{n-1-k}$. This implies
existence of an open set $U $ of regular points containing $(x, \xi)$ and an open dense set of compact regular orbits 
( (see Lemma \ref{OD} and  Proposition \ref{NONCPTPROP}).

 In this case there exist 
angle coordinates on ${\mathbb T}^k$ and linear coordinates $y$ on $\R^{n-1-k}$  so that the flow is equivalent to $G^t(I, \phi, I', y) = (I, \phi + t \omega_I, y + t \omega_{y} )$ for some $\omega_I \in \R^k, \omega_{y} \in \R^{n-1-k}$.
The $\omega$ (and $\alpha$) limit set of  the geodesic orbit $G^t(x, \xi)$ lies in a  compact $G^t$-invariant  set, and every point of intersection with $S^*_{x_0} \ecal_A$ is a singular point. Recall that the limit set is compact, connected and invariant.

\end{itemize}

We now prove Proposition \ref{PSFSING}. 
\begin{lem}\label{REG}  Suppose that $x_0$ is a self-focal point of a multi-axial  ellipsoid. Suppose that the set $(S^*\ecal_A)_{reg, c}$ of
compact regular orbits is open dense (Lemma \ref{OD}). Suppose that the geodesic flow is isoenergetically non-degenerate. Then,  for any $(x_0, \xi) \in S^*_{x_0} M$, there does not exist $\xi \in S^*_{x_0} \ecal_A$ such that $\R^{n-1} \cdot (x_0, \xi)$ is a compact torus of dimension $n-1$. \end{lem}

In combination with  Proposition \ref{ISOENDPROP2}, Lemma \ref{REG} implies, 

\begin{cor} \label{PFSINGCOR} If $x_0$ is a self-focal point of a multi-axial  ellipsoid, then  for any $(x_0, \xi) \in S^*_{x_0} M$, there does not exist $\xi \in S^*_{x_0} \ecal_A$ such that $\R^{n-1} \cdot (x_0, \xi)$ is a compact torus of dimension $n-1$. 

\end{cor} 
 Proposition \ref{PSFSING} follows from Corollary \ref{PFSINGCOR}. Hence, to complete the proof it suffices to prove Lemma \ref{REG}.

\begin{proof}

 If there exists $(x_0,\xi) \in (S^*\ecal_A)_{reg, c} \cap S^*_{x_0} \ecal_A$ then $ (S^*\ecal_A)_{reg, c} \cap S^*_{x_0} \ecal_A$ is open 
dense in $S^*_{x_0} \ecal_A$.  We may introduce action-angle variables $(I, \phi)$ in an open neighborhood $U$ of $(x_0, \xi)$.   For each point $(x_0, \xi)$ in this set,  we consider the $G^t$ orbit $G^t(x_0, \xi)$ in the compact torus
orbit $T_I = \R^{n-1}(x_0, \xi)$. It lies in the intersection $\Lambda_{x_0}
\cap \R^{n-1} (x_0, \xi).$ Since  $\Lambda_{x_0} \simeq S^1 \times S^{n-2}$ and  $ \R^{n-1} (x_0, \xi)$ has a different topology, 
the intersection has codimension $\geq 1$ in both Lagrangian submanifolds. Therefore the geodesic
cannot be dense in the compact torus orbit $\R^{n-1} (x_0, \xi).$ 

Let $\omega_I = \nabla_I H : U \to \R^{n-1}$ be the frequency vector map. The geodesic flow is given on $T_I$ by 
$G^t(I, \phi) = (I, \phi + t \omega_I)$. Since the orbit is not dense, $\omega_I $ is resonant for all points in $U$. Resonance means that, for each
$(x_0, \xi) \in U$, there exists
$\vec k \in \Z^{n-1}$ so that $\langle \omega_I(x_0, \xi) , \vec k \rangle =0$. Since $\vec k \in \Z^{n-1}$ and $\Z^{n-1}$ is totally disconnected, it must be the case that
there exists $\vec k \in \Z^{n-1}$ so that  $\langle \omega_I(x_0, \xi) , \vec k \rangle =0$ for all $(x_0, \xi) \in U$. This contradicts iso-energeticity.

\end{proof}

\subsection{Proof of Propositions \ref{SINGLEAF} and \ref{DEGPROP} }

We assume throughout this section that $\ecal_A$ is multi-axial. We now prove Propositions \ref{SINGLEAF} and \ref{DEGPROP}.

\subsubsection{Proof of Proposition   \ref{SINGLEAF} }

\begin{proof}   Since the orbits in an open dense set of $S^* \ecal_A$ are Lagrangian tori, 
an open dense set
of points of $S^*_H M$ have regular compact orbits. The intersections of these orbits with $S^*_H M$  are Lagrangian submanifolds which we refer to
as the `regular c-leaves' of the reduced foliation of  $S^*_H M$.

The regular $c$-leaves  cannot intersect $S^*_{x_0} \ecal_A$ by Proposition \ref{REG}, i.e. the regular $c$-leaves are disjoint from $S^*_{x_0} \ecal_A$.  If $S^*_{x_0} \ecal_A$ is {\it not}
a leaf of this foliation, then it must intersect some leaves of the foliation and in particular must intersect a regular compact leaf (because they are open dense).
By Proposition \ref{REG}, this cannot happen. Hence, $S^*_{x_0} \ecal_A$ must be a singular leaf of the foliation. 
\end{proof}

\begin{lem}\label{SINGLEVEL}
  The eigenvalue moment map is constant on $S^*_{x_0} \ecal_A$.

\end{lem}

\begin{proof} Since $S^*_{x_0} \ecal_A$ is a leaf of the reduced foliation by Lemma \ref{SINGLEAF}, the eigenvalue moment map is constant on $S^*_{x_0} \ecal_A$.
\end{proof}
This completes the proof of Proposition \ref{SINGLEAF}.

\subsubsection{Proof of Proposition \ref{DEGPROP}}

\begin{proof}  If  the eigenvalues of $L(x_0, \xi)$ correspond to  non-degenerate quadrics,  we get the contradiction
that  $S^*_{x_0} \ecal_A$ spans $T_{x_0} \ecal_A$ and this tangent space cannot be tangent to a single non-degenerate quadric.  Hence,
every one of the quadrics must be degenerate. Apparently, this can only happen when every $\lambda_j $ is an eigenvalue of $A$, which forces
the points to have $x_j =0$ for $n-2 $ j's. Hence $x_0$ must lie on the intersection of $\ecal_A$ with $n-2$ hyperplanes, i.e. be a sub-ellipse (and
closed geodesic).

\end{proof}

\section{Proof of Theorem \ref{THEODISTINCT}}
In this section, we prove the main result (Theorem \ref{THEODISTINCT})  for ellipsoids with distinct axes.

\begin{proof}

By Lemma \ref{SINGLEVEL}, the moment map is constant on $S^*_{x_0} \ecal_A$.  By Proposition \ref{DEGPROP}, 
the quadrics associated to the eigenvalues must all be degenerate.  By  Lemma \ref{JACINT}, Corollary \ref{lambdaalpha} and Lemma
\ref{AUDINLEM}, for all $j$, $\lambda_j = a_j$ and $x_j =0$ for $n-2 $ values of $j$. That is, $x_0$ must lie on one of the coordinate
ellipses obtained by slicing with the codimension two plane $x_j = 0$ in $\R^n$. We recall here that $A$ has been diagonalized, so that
its eigenvectors are the standard basis $e_j$ of $\R^n$. 
We now argue that no such $x_0$ can perfect focal points.  

Let us denote the index of the one non-zero component of $x_0$ by $j$. 
Then, consider any three-dimensional subspace  $\R^3_{i, j, k }  = \rm{Span}\{e_i, e_j, e_k \}  \subset \R^n$ and the corresponding
totally geodesic two-dimensional ellipsoid $\ecal_{a_i, a_j, a_k}$  obtained by slicing $\ecal_A$ with $\R^3_{i, j, k }$. Then $x_0$ must be a perfect focal
point of $\ecal_{a_i, a_j, a_k}$. Hence , $x_0$ must be an umbilic point of each $\ecal_{a_i, a_j, a_k}$  as $(j,k)$ range over pairs
of distinct indices $j < k$ with $j > i$. 

We then recall that if  $A = (a_i^2, a_j^2,a_k^2) $ with $0 < a_k < a_j < a_j$, then  an  umbilic point is  given by
\begin{equation} \label{UMBILIC} x =
 (\pm \sqrt{\frac{a_i^2 - a_j^2}{a_i^2 - a_k^2}}, 0, \pm  \sqrt{\frac{a_j^2 - a_k^2}{a_i^2 - a_k^2}}) = \pm \sqrt{\frac{a_i^2 - a_j^2}{a_i^2 - a_k^2}} e_1 + \pm  \sqrt{\frac{a_j^2 - a_k^2}{a_i^2 - a_k^2}} e_3.  \end{equation}
 In particular, $x$ must lie on the middle  coordinate ellipse where $x_2= 0$ so  $a_j$ cannot be the lowest or highest eigenvalue. But the main point is that $x_0$ must be umbilic for
 all choices of $a_i, a_k$ and therefore the expressions for $x$ can only differ by a sign as $a_i, a_k$ vary. Thus, $a_k$ can be any
 eigenvalue $< a_j$ and $a_i$ can be any eigenvalue $> a_j$. 
 
 A little bit of algebra shows
 that this cannot be true if the $a_j$ are distinct. To make this clear we assume write two triples of $(a_i, a_j, a_k)$ with the same $a_j$
 and denote the second by $(b_i, a_j, b_k)$. The eigenvalues are understood to be ordered in decreasing order. 
 Squaring gives the equations,
 $$\frac{a_i^2 - a_j^2}{a_i^2 - a_k^2} = \frac{b_i^2 - a_j^2}{b_i^2 - b_k^2},\;\;  \frac{a_j^2 - a_k^2}{a_i^2 - a_k^2}=  \frac{a_j^2 - b_k^2}{b_i^2 - b_k^2}. $$
 This pair of equations must hold for $a_j$ when $a_i > a_j > a_k$ range over the other eigenvalues $b_i > b_j < b_k$ of $A$. Note
 that we may choose  one of $b_i, b_k$ so that it equals the corresponding $a_i$ resp. $a_k$. 
  
 There is no loss of generality in assuming that $a_j  =1$ since, we may dilate the ellipse without changing the problem. Moreover,
 we simply notation by writing $\alpha_j = a_j^2, \beta_j = b_j^2$. 
  Thus, we get 
  $$\frac{\alpha_i -1}{\alpha_i - \alpha_k} = \frac{\beta_i - 1}{\beta_i - \beta_k},\;\;  \frac{1 - \alpha_k}{\alpha_i - \alpha_k}=  \frac{1 - \beta_k}{\beta_i - \beta_k}. $$
  Taking a ratio of the two left sides and equating it to the ratio of the two right sides gives,
   $$\frac{\alpha_i -1}{1- \alpha_k} = \frac{\beta_i - 1}{1- \beta_k}. $$
 If we  choose $\beta_i, \beta_k$ so that $\beta_i = \alpha_i$, then  $\beta_k = \alpha_k$. Vice-versa if we choose the axes so that
 $\beta_k = \alpha_k$ then $\beta_i = \alpha_i$. It follows that the system of equations forces all eigenvalues $< \alpha_j$ to be equal
 and all eigenvalues $> \alpha_j$ to be equal, contradicting that all eigenvalues are distinct. 
 
\end{proof}

\subsection{\label{MULTASECT} Second Proof of Theorem \ref{THEODISTINCT}}
Combining Lemma \ref{SINGLEVEL} with Moser's spectral interpretation of the Hamiltonian seems to give powerful conclusions, including a proof of Theorem \ref{THEODISTINCT}.
 
Suppose $A$ has distinct eigenvalues, and $x_0\in\mathcal{E}_A$ is a self-focal point. From Lemma \ref{SINGLEVEL}, together with the spectral definition of the moment map in \eqref{pcaldef}, it follows that the $n-2$ nontrivial eigenvalues of $L(x_0,\xi)$ and their multiplicities are preserved as $\xi\in S^*_{x_0}\mathcal{E}_A$ varies.
Now fix some $\xi_0$ and choose an orthonormal eigenbasis $(\phi_1,\ldots,\phi_{n-2},\xi,\mathbf{n})$ for $L(x_0,\xi_0)$. We denote the $n-2$ nontrivial eigenvalues by $\lambda_j$ as throughout this section.

If we analytically continue the corresponding eigenvalue and eigenvectors along a curve of $\xi$ on $S^*_{x_0}\mathcal{E}_A$ through $x_0$, we must then have $\dot{\lambda}_j=0$ for every $j$ along this curve.

From Lemma \ref{FIRSTVAR}, it follows that for each $j$ and each $\dot\xi\in T_{\xi_0}T^*_{x_0} \mathcal{E}_A$, we have
\begin{equation}
\langle \phi_j,\dot\xi\rangle \langle (A-x_0\otimes x_0)\phi_j,\xi_0\rangle=0.
\end{equation}
As every $\phi_j$ is orthogonal to both $\mathbf{n}$ and $\xi$, and hence in $T_{\xi_0}T^*_{x_0} \mathcal{E}_A$, we can put $\dot\xi=\phi_j$ to obtain 
\begin{equation}\label{VARVANISH}
\langle (A-x_0\otimes x_0)\phi_j,\xi_0\rangle=0.
\end{equation}
Hence 
\begin{equation}
(A-x_0\otimes x_0)\phi_j=P_{\xi_0}(A-x_0\otimes x_0)P_{\xi_0}\phi_j=L\phi_j=\lambda_j\phi_j
\end{equation}
for $j=1,\ldots, n-2$, that is, the $\phi_j$ are also eigenvectors of $A-x_0\otimes x_0$ with the same eigenvalue.
Furthermore, we have
\begin{eqnarray}
& &(A-x_0\otimes x_0)\xi_0 \\
&=& \sum_{j=1}^{n-2}\langle (A-x_0\otimes x_0)\xi_0,\phi_j\rangle \phi_j+ \langle(A-x_0\otimes x_0)\xi_0,\xi_0\rangle \xi_0+\langle(A-x_0\otimes x_0)\xi_0,\mathbf{n}\rangle \mathbf{n}\\
&=& \sum_{j=1}^{n-2}\langle \xi_0,(A-x_0\otimes x_0)\phi_j\rangle \phi_j+ \langle(A-x_0\otimes x_0)\xi_0,\xi_0\rangle \xi_0+\langle \xi_0,(A-x_0\otimes x_0)\mathbf{n}\rangle \mathbf{n}\\
&=& \sum_{j=1}^{n-2}\lambda_j \langle \xi_0,\phi_j\rangle \phi_j+ \langle(A-x_0\otimes x_0)\xi_0,\xi_0\rangle \xi_0\\
&=& \langle(A-x_0\otimes x_0)\xi_0,\xi_0\rangle \xi_0
\end{eqnarray}
and so $\xi_0$ is an eigenvector of the fixed operator $A-x_0\otimes x_0$ for any $\xi_0\in S_{x_0}^* \mathcal{E}_A$. This is only possible if $A-x_0\otimes x_0$ restricts to a scalar operator on the subspace $T_{x_0}^* \mathcal{E}_A$, and so $A-x_0\otimes x_0$ must have an eigenspace of dimension $n-1$. Taking the intersection of this eigenspace with $x_0^\perp$, we obtain an eigenspace for $A$ of dimension at least $n-2$. In particular, this is impossible for $n\geq 4$ if $A$ has distinct eigenvalues.

\section{\label{MYMULTSECT} Reduction of Theorem \ref{THEO4DISTINCT} to Theorem \ref{THEODISTINCT}}

In this section, we begin the reduction of Theorem \ref{THEO4DISTINCT}, to the multiaxial case, Theorem \ref{THEODISTINCT}.

Consider a symmetric  matrix $A$ on $\R^{n}$  with multiple eigenvalues $\alpha_1 \leq \alpha_2 \leq \cdots \leq \alpha_{n}$. We denote
the distinct eigenvalues by $\alpha_j^*$, the number of distinct eigenvalues by $n_*$ or $r$, and the multiplicity of $\alpha_j^*$ by $m(\alpha_j^*)$. Thus,
$\sum_{j=1}^r m_j = n$.  We consider the 
quadratic form $q_A(x) = \langle A^{-1} x, x \rangle $ on $\R^n$. 

\begin{defn} The orthogonal group $O(q_A)$ of $q_A$ is $G = SO(m(\alpha_1^*)) \times \cdots \times SO(m(\alpha_n^*))$, i.e. $T_g A T_g^* = A$. \end{defn}  Obviously,  $G$ acts on the
ellipsoid
$\langle A^{-1} x, x \rangle = 1$.
\begin{lem}
The group $G = SO(m_1) \times \cdots \times SO(m_r)$ acts on $\ecal_A$ by isometries and  $G$ commutes
with $A$. \end{lem} \begin{proof} To prove the first statement it suffices to consider a factor $SO(m_j)$. It acts only in the $j$th block
of indices by rotations of the corresponding coordinates.  If we fix the other coordinates, we obtain a slice of the ellipsoid which is isometric
to a sphere. The rotations in the block are isometries of the slice. 
The second follows because from the fact that $g \in O(q_A)$.   
\end{proof}

\subsection{Lax matrix and isospectral deformation} We view $L(x, \xi)$ for $(x, \xi) \in S^*\ecal_A$  as an $n \times n$ matrix.  Hence it makes sense to consider the conjugates $g L(x, \xi) g^{-1}$
with $g \in G$. 
We lift the $G$ action to $T^*\ecal_A$ in the canonical way and denote the lifted action by $\tau(g) (x, \xi) = (g x, (D_x g)^{-1 tr} \xi). $

\begin{prop}\label{gL} $L(\tau(g)(x, \xi)) = g L(x, \xi) g^{-1}$, hence the Lax matrices along   a $G$-orbit  $G \cdot (x, \xi)$ in $S^* \ecal_A$ form an isospectral class. \end{prop}

\begin{proof} By definition, $L(x, \xi) = P_{\xi} (A + x \otimes x^* ) P_{\xi}$ and $$\begin{array}{lll} g L(x, \xi) g^{-1} & = &  g P_{\xi} g^{-1} (g A g^{-1} + g x \otimes (gx)^*) 
g P_{\xi} g^{-1} \\ && \\
& = & g P_{\xi} g^{-1} ( A + g x \otimes (gx)^*) 
g P_{\xi} g^{-1} \\ && \\
& = &  P_{\tau(g) (x, \xi)}  ( A + g x \otimes (gx)^*) 
P_{\tau(g)(x, \xi)}\\ && \\
& = &  L(\tau(g)(x, \xi)) .  \end{array}$$
The second to  last identity follows from the fact that  $g P_{\xi} g^{-1} = P_{\tau(g)(x,  \xi)}$, where we identify vectors and co-vectors using the Riemannian metric.
Since $g$ acts by isometries,  $g P_{\xi} g^{-1}$ is the orthogonal projection onto the orthogonal complement of $g \xi = D_x T_g \xi$. 

\end{proof}

\subsection{Isotropy groups} Let $x \in \ecal_A$ and let $G_x =\{g \in G: g x =x \}$ be its isotropy group. For $g \in G_x$, 
$D_x T_g: T_x \ecal_A \to T_x \ecal$ is an isometry. 

\begin{lem} \label{ISOTROPLEM} 
If $g \in G_x$ then $T_g (A + x \otimes x) T_g^{-1} = (A + x \otimes x)$. 
\end{lem}

Even if $T_g A T_g^{-1} = A$, we need to understand $T_g (A + x \otimes x^*) T_g^{-1} = A + (T_g x) \otimes (T_g x)^*$. Further
we need to understand $(T_g x, (DT_g)^{* -1} \xi)$.

Let $G_{x, \xi} = \{g \in G: D_x T_g \xi = \xi\}$ be the isotropy group of this
linear transformation.

\begin{lem} \label{ISOTROPLEM2} 
 If $\dim G_x $ is even, then $G_{x, \xi} = \{I\}$ for all $x \in T_x \ecal_A$ while if $\dim G_x$ is odd, then
each $D_x T_g$ has a `axis of rotation' and $G_{x, \xi}$ is the rotations of which $\xi$ is the axis. In this case, $T_g L(x, \xi) T_g^{-1} = L(x, \xi)$.
\end{lem}

\section{\label{4AXES} Ellipsoids with at least four distinct axis lengths:   Proof of Theorem \ref{THEO4DISTINCT} }

In Section \ref{MULTASECT}, we proved that the ellipsoid $\SE_A$ has no self-focal points if $A$ has distinct eigenvalues, and $n\geq 4$. In this section, we use this result to show that there cannot exist self-focal points in  the case of degenerate ellipsoids with at least four distinct axis lengths.

\begin{lem}
	\label{lem:curveoffoci}
	If $G$ is a Lie group acting isometrically on a Riemannian manifold $M$ and $p\in M$ is a self-focal point, then the orbit $Gp\subset M$ consists entirely of self-focal points.
\end{lem}
\begin{proof}
	The geodesic flow commutes with the Lie group action lifted to $T^*M$, and so for $g\in G$ and self-focal $x\in M$, we have $g\cdot x=g\cdot \exp_{x}(T\xi)=\exp_{g\cdot x}(Tg_*(\xi))$ for arbitrary $\xi$, and so $g\cdot x$ is self-focal.
\end{proof}
We are now ready to 
prove Theorem \ref{THEO4DISTINCT}.

\begin{proof}
	Suppose $A=\diag(\alpha_j)_{j=1}^n$ is a diagonal $n\times n$ matrix with $l\leq n$ distinct eigenvalues. Reordering coordinates if necessary, we can assume that $\alpha_1,\ldots,\alpha_l$ are distinct.
	Now suppose that $x\in \SE_A\subset \RR^n$ is a self-focal point with first return time $T>0$.
	Taking $V_j=\ker(A-\alpha_j I)$, we choose orthogonal matrices $B_1,\ldots B_l$ so that $B_j(x)=\|\pi_{V_j}(x)\|e_j+\pi_{V_j^\perp}(x)$. Application of Lemma \ref{lem:curveoffoci} using the isometry $B=\prod_{j=1}^l B_j\in O(n)$ then implies that there is a self-focal point $y=Bx=(y_j)_{j=1}^n$, with $y_j=0$ for $j>l$.
	
	This self-focal point $y$ lies on the sub-ellipsoid $\tilde{\SE}_A=\{x\in \SE_A: x_j=0\textrm{ for }j>l\}$. The sub-ellipsoid $\tilde{\SE}_A$ is an isometrically embedded copy of the $l$-axial ellipsoid $\SE_{\tilde{A}}$ determined by the matrix $\tilde{A}=\diag(\alpha_1,\ldots,\alpha_l)$. We denote the 
	isometric embedding by $\iota:\SE_{\tilde{A}}\rightarrow\tilde{\SE}_A$.
	
	The geodesics $\exp_y(t\xi)$ on $\SE_A$ with $\xi\in S_y^*\SE_A$ and $\xi_j=0$ for $j>l$ then correspond bijectively to geodesics on $\SE_{\tilde{A}}$ through $\iota^{-1}(y)$, and this correspondence is a bijection. As the $\exp_y(t\xi)$ have common first return time $T$, the isometric embedding $\iota:\SE_{\tilde{A}}\rightarrow\tilde{\SE}_A$ yields the existence of a self-focal point $\iota^{-1}(y)\in \SE_{\tilde{A}}$. From Theorem \ref{THEODISTINCT}, this is only possible if $l<4$.
\end{proof}

\section{\label{3AXESSECT} Ellipsoids with exactly three distinct axes} 
In this section we prove Proposition \ref{3AXESPROP}, and in addition we prove further results in the triaxial case. 
We denote the multiplicities of the axes by $(m_1, m_2, m_3)$ (with  $m_1 + m_2 + m_3 = n$), so that $$\ecal_A = \{(\vec x_1, \vec x_2, \vec x_3) \in \R^{m_1} \oplus \R^{m_2} \oplus \R^{m_3} :\;  \frac{||\vec x_1||^2}{a_1^2} +  \frac{||\vec x_2||^2}{a_2^2} +  \frac{||\vec x_3||^2}{a_3^2} =1\} .$$
Associated to the ellipsoid is a triaxial  two-dimensional ellipsoid $$\ecal_{\vec a} = \{(x_1, x_2, x_3) \in  \R^3: \sum_{j=1}^3 \frac{x_j^2}{a_j^2} = 1\}. $$

 For any choice of indices $i_{j, k}$ with $j= 1, 2, 3$ and $k = 1, \dots, j$, let $e_{j,k}$ be the standard basis of $\R^{m_j}$, and define the totally geodesic   coordinate
 embeddings,
 $$\iota_{k_1, k_2, k_3}: \ecal_{a_1, a_2, a_3} \to \ecal_A, \;\; \iota_{k_1, k_2, k_3}(x_1, x_2, x_3) = x_1 e_{1, k_1} + x_2 e_{2, k_2}
 + x_3 e_{3, k_3}. $$
 There are $m_1 m_2 m_3$ such coordinate  embeddings.
  
An umbilic point of $\ecal_{a_1, a_2, a_3} \subset \R^3$ has the form, $(x_1, 0, x_3)$ with \begin{equation} \label{UMBEQ} x_1 = \pm a_3 \sqrt{\frac{a_2^2 - a_3^2}{a_1^2 - a_3^2}},  
x_3 = \pm a_1 \sqrt{\frac{a_2^2 - a_1^2}{a_3^2 - a_1^2}}. \end{equation}

 Let $G= SO(m_1) \times SO(m_2) \times SO(m_3) \times \Z_2^n$ denote  the isometry group of $\ecal_A$. Here $\Z_2^n$ is the group
 of sign changes $x_j \to - x_j$. Then $G$ acts on the coordinate embeddings of $\ecal_{\vec a}$, producing more totally geodesic
 two-dimensional ellipsoids. 
 To consider the `set' of totally geodesic embedded two-dimensional ellipsoids, we  consider the isotropy groups of the umbilic points of the coordinate embedded
 $\iota_{\vec k}(\ecal_{\vec a})$. 
With no loss of generality, let us assume $k_1 = 1, k_3 = 1$, so that the embedded umbilic point has the coordinates
 $(x_1, 0, 0, \cdots, 0; \vec 0; x_3, 0, 0 \cdots 0). $
 When $m_1, m_3$ are odd, the  isotropy group of this point is $G_{umbilic} = SO(m_1 -1) \times SO(m_2) \times SO(m_3-1)$, where the subgroups rotate around
 the $e_1$ axes of $\R^{m_1}$ resp. $\R^{m_2}$. When a multiplicity is even, the isotropy group of that factor is   of one lower dimension.
 The orbit of the umbilic point under the symmetry group has the form, 
 $$S^{m_1 -1} \times \{0\} \times S^{m_3 -1} $$
 when both multiplicities are odd.


 The following is obvious:
 \begin{lem} Suppose that $x_0 \in \ecal_A$ is a self-focal point that lies on a totally geodesic embedded ellipsoid. Then $x_0$ is self-focal
 on the sub-ellipsoid. In particular, if the subellipsoid is two dimensional, $x_0$ must be an umbilic point of the subellipsoid.\end{lem}

 As the Lemma suggests, the aim is to constrain the possible self-focal points of $\ecal_A$ by relating them to umbilic points of totally geodesic embedded ellipsoids. The
 set of such ellipsoids depends on the multiplicites $(m_1, m_2, m_3)$ and the associated isometry group.
 
  
 \subsection{Tri-axial ellipsoids with extreme multiplicities} 
 In this section, we restrict attention to the extreme cases of $(m_1, m_2, m_3)$ where two multiplicities equal $1$, i.e. with $(m_1, m_2, m_3) =  (n-2, 1,1), (1, n-2, 1), (1, 1, n-2)$.

There are essentially  two cases:

\begin{itemize}

\item (i)\; $m_1 = m_3 = 1, m_2 = n-2.$ I.e. $(m_1, m_2, m_3) =  (1, n-2, 1)$  \bigskip

\item (ii)\; $m_1= n-2, m_2= m_3 = 1$, or $m_1 = m_2 = 1, m_3 = n-2$;  \bigskip


\end{itemize}

 \begin{prop} In case (i), $\ecal_A$ posseses a non-polar self-focal point $u$. 
 \end{prop}
 
 \begin{rem} Interestingly, it is also an umbilic point of $\ecal_A$ in the sense that the second fundamental form at $u$
 is a multiple of the metric \cite{LSG}.  We are not aware of any apriori identification of self-focal points and umbilic points on ellipsoids.
 \end{rem}
 
 \begin{proof}

We are considering the case where the middle axis has multiplicity $n-2$. The possible coordinate embeddings $\iota_{\vec k}$
 of $\ecal_{\vec a}$ have a fixed $k_1 = e_1, k_n = e_n$ and only the middle axis $k_2$ is variable. These embeddings are the
 orbit of any fixed embedding under the identity component $G_e =\{1\} \times SO(n-2) \times \{1\}$ of the isometry group of $\ecal_A$, which fixes $e_1$ and $ e_n$ and rotates the second axis in the orthogonal complement
 of $\rm{Span}\{e_1, e_n\}$.
 
  Let $\wt u = (x_1, 0, x_2) \in \ecal_{\vec a} $ be 
 an umbilic point, as in \eqref{UMBEQ}. The image of $\wt u$ under the embedding
$\iota_{1, k_2, n}$ is 
$$u = \iota_{1, k_2, n}(x_1, 0, x_3) = x_1 e_{1} 
 + x_3 e_{n} = (x_1, \vec 0, x_3). $$
Thus, the image of $\wt u$ is the same $u$  for every $\iota_{\vec k}$ in this case. The point $u$ is fixed
 under $G_e =  \{1\} \times SO(n-2) \times \{1\}$, the identity  component of the isometry group, its   isotropy group $ G_u:= \{g: gu = u\} = G_e =  \{1\} \times SO(n-2) \times \{1\}$.

 We claim that $u$ is a self-focal point of $\ecal_A$. Certainly, this
 is true for the directions in $S^*_u \ecal_A$ lying in  $T^*_u (\iota_{\vec k}(\ecal_{\vec a}))$ for any $\vec k= (1, \vec k_2, n)$. To prove this, we decompose
 $T_u \ecal_A$ into irreducible subspaces for $G_u$.
 
Let $\vec n_u$ be the unit outward normal to $\ecal_A$ at $u$ in $\R^n$. Then,  $T_u \ecal_A \oplus \R \vec n_u = T_u \R^n$. Also, 
$\R^n = \R e_1  \oplus \R^{n-2} \oplus \R e_n$. The isotropy group $G_u= SO(n-2) $ acts on the middle $\R^{n-2}$. 

  Let  $\gamma_u(t)$   be the $\iota $ applied to the hyperbolic geodesic $\wt \gamma(t)$  through $\wt u \in \ecal_{\vec a}$. Then,
  $\wt \gamma(t) = (x_1 \cos t, 0, x_3 \sin t) $ where $x_1, x_3$ are as in \eqref{UMBEQ}, and $\gamma_u(t) = (x_1 \cos t + x_3 \sin t, \vec 0,  - x_1 \sin t + x_3 \cos t) $, and
  $\dot{\gamma}_u(0) = (x_3, \vec 0, - x_1). $

Also, let 
  $J \dot{\gamma}_u(0)$ be the image under $\iota_*$ of the rotate $J \dot{\wt{\gamma}}(0)$  by $\frac{\pi}{2}$ of $\frac{d}{dt} \wt \gamma(t) |_{t=0}$. A normal at the umbilic 
  point is $(\frac{x_1}{a_1^2}, \frac{x_2}{a_2^2}, \frac{x_3}{a_3^2}) |_u = (\frac{x_1}{a_1^2}, 0, \frac{x_3}{a_3^2})$ where (again) $x_1, x_2$ are given in \eqref{UMBEQ}.
   Since $x_2 = 0$ on $\wt \gamma$, $J \dot{\wt{\gamma}}(0) = (v_1, v_2, v_3)$ has a non-zero $v_2$ coordinate, while $(v_1, 0, v_3) \cdot (x_3, \vec 0, - x_1) = 0$
   and $(v_1, 0, v_3)  \cdot  (\frac{x_1}{a_1^2}, 0, \frac{x_3}{a_3^2})= 0. $ A little calculation shows that the equations force $v_1 = v_3 =0$. Hence, $J \dot{\gamma}_u(0)
   = e_{2,1} $.  It follows that $T_u \iota(\ecal_{\vec a}) = \rm{Span} \{(x_3, \vec 0, - x_1), (0, (1, 0, \cdots, 0), 0)\}$. 
   
Every geodesic with initial tangent vector in $T_u \iota(\ecal_{\vec a})$ loops back to $u$ at time $2 \pi$. Every other geodesic from $u$ has initial tangent
vector in $V_u$, and therefore must also loop back in time $2 \pi$. This concludes the proof that $u$ is a self-focal point of  $\ecal_A$.

Not only is $u$ self-focal, but it is also not a pole. Indeed, it is not a pole in the subellipsoid $\iota(\ecal_{\vec a}). $
    
  \end{proof}
  
\subsubsection{Case (ii)} The remaining case (ii) is more complicated and the question  whether it has
self-focal points is left open. The question is  open even when $n =4$ and $\ecal_A$ is a three-dimensional ellipsoid with axis
multiplicities $(2,1,1)$. This is a key test of existence of self-focal points, because a necessary condition that there are self-focal points in dimension $n$
is that there exist self-focal points when $n=4$.   However, in view of the $G = SO(2)$ symmetry group, self-focal points in this case are never
fully `twisted', i.e. they come in one-parameter families.

 The geodesic flow on three-dimensional $(2,1,1)$  ellipsoids is studied in \cite{DD07} and is shown to be Liouville integrable. One of the integrals
 is the angular momentum $J$ corresponding to the $SO(2)$-action. 
It is  also shown in \cite{DD07}  that the symplectic reduction with respect to $J$ of the geodesic flow  of the $(2,1,1)$ ellipsoid $\ecal$ is the Rosochatius Hamiltonian flow on the  two-dimensional
ellipsoid, i.e. the flow of the reduced Hamiltonian  \begin{equation} \label{ROSOHAM} H = |\xi|_g^2 + \frac{j^2}{2 x_1^2}, \;\; \rm{on} \; (\{J = j\}/SO(2))  \end{equation} The orbits of
 the flow with $j \not= 0$ are constrained to lie on one side of $x_1 = 0$.  Those with $j = 0$ are the usual geodesics of $\ecal_{\vec a}$.
 Rosochatius systems on the sphere are discussed in \cite{M80}, and on the two-ellipsoid 
 \cite{DD07,Jo12}.  One has the following 
  \begin{lem} \label{ROSOLEM}  If there exists a self-focal point $u = (\vec x, x_2, x_3) \in \ecal_A$,  then the umbilic points $(\pm x_1, 0, 
\pm x_3)$ \eqref{UMBILIC}  of  $\ecal_{\vec a}$ must be
 self-focal for the  Rosochatius flows for every value of $j$.   \end{lem}

 Little seems to be known about Rosochatius flows.  The dynamics depends significantly
 on the value of $j$.  It would take us far afield to study them, and we leave  the question
 whether umbilic points are in fact self-focal for every Rosochatius flow for future investigation.

\end{document}